\documentclass[reqno,12pt]{amsart}

\usepackage[margin=2.4cm]{geometry}

\usepackage[utf8]{inputenc}
\usepackage[english]{babel}
\usepackage[all]{xy}
\usepackage{amsmath,latexsym,amssymb,verbatim}
\usepackage{mathrsfs}
\usepackage[export]{adjustbox} 
\usepackage{ifthen}
\usepackage{etoolbox}

\usepackage{cite}

\usepackage{hyperref}

\usepackage{amsfonts}

\usepackage{amsthm}
\usepackage{booktabs}
\usepackage{graphicx}
\usepackage[dvipsnames]{xcolor}

\usepackage{mathtools}
\usepackage{subcaption}
\usepackage{faktor}
\usepackage{marginnote}
\usepackage{tikz-cd}
\usepackage{stmaryrd}
\usepackage{centernot} 
\usepackage{enumitem}

\usepackage{blkarray}
\usepackage{multirow}
\usepackage{scalerel}
\makeindex

\newcommand{\plim}[1]{\,\underset{#1}{\underset{\longleftarrow}{\lim}}\,}
\newcommand{\face}{{\trianglelefteqslant}}

\definecolor{azuldibujos}{RGB}{179, 230, 230}
\definecolor{contrast2}{HTML}{E50057}
\definecolor{contrast3}{HTML}{FFDD07}
\definecolor{H0}{HTML}{c8dfe3}
\definecolor{H1}{HTML}{5e8b94}

\newcommand{\R}{\mathbb{R}}

\newcommand{\cF}{{\mathcal F}}
\newcommand{\cG}{{\mathcal G}}

\newcommand{\cM}{\mathcal{M}}
\newcommand{\cN}{\mathcal{N}}

\newcommand{\cP}{\mathcal{P}}

\newcommand{\bbN}{{\mathbb N}}

\DeclareMathOperator{\Sh}{\mathbf{Sh}}
\DeclareMathOperator{\Fun}{\mathbf{Fun}}

\DeclareMathOperator{\Id}{{Id}}

\DeclareMathOperator{\vect}{\mathbf{vect}_\Bbbk}

\newcommand{\pos}{\mathbf{pos}}

\DeclareMathOperator{\End}{End}
\DeclareMathOperator{\img}{im}

\DeclareMathOperator{\coker}{coker}
\DeclareMathOperator{\sign}{sign}

\DeclareMathOperator{\op}{op}

\DeclareMathOperator{\st}{st}

\newcommand{\h}{\mathbf{h}}

\newtheorem{theorem}{Theorem}[section]

\newtheorem{proposition}[theorem]{Proposition}

\newtheorem{corollary}[theorem]{Corollary}

\newtheorem{lemma}[theorem]{Lemma}

\theoremstyle{definition}
\newtheorem{definition}[theorem]{Definition}

\newtheorem{example}[theorem]{Example}

\newtheorem{remark}[theorem]{Remark}

\AtBeginEnvironment{example}{\pushQED{\qed}}
\AtEndEnvironment{example}{\popQED}

\AtBeginEnvironment{ex}{\pushQED{\qed}}
\AtEndEnvironment{ex}{\popQED}

\begin{document}

\title[Intervals for Multipersistence and Sheaf Data]{Interval Decompositions for Multipersistence Modules over Finite Posets and Robustness of Sheaf Data on Simplicial Complexes}

\date{\today}



\author[P. Hern\'andez-Garc\'ia]{Pablo Hern\'andez-Garc\'ia}
\author[D. Hern\'andez Serrano]{Daniel Hern\'andez Serrano}
\author[D. S\'anchez G\'omez]{Dar\'io S\'anchez G\'omez}

\address{Departamento de Matem\'aticas and Instituto Universitario de F\'isica Fundamental y Matem\'aticas (IUFFyM), Universidad de Salamanca, Salamanca, Spain}
\email{pablohg.eka@usal.es, dani@usal.es, dario@usal.es}

\begin{abstract}
We prove structure theorems for multipersistence modules indexed by finite posets that are not totally ordered. Specifically, we consider pointwise finite-dimensional modules over the opposite of the poset of non-empty subsets of a finite set, and give sufficient conditions, expressed through transition morphisms, for such modules to split as direct sums of interval modules. In the general case, the interval summands and multiplicities are explicitly determined by dimensions at finitely many indices. Although the assumptions may look algebraically restrictive, we show that they have a natural geometric origin in a robustness theory of cellular sheaf data over simplicial complexes, where one studies how algebraic information, compatibility constraints, and cohomological obstructions persist under structural failures. We extend thickness and cohesion from simplicial cohomology to cellular sheaves: thickness detects the dependence of cohomology classes on high-dimensional support, while cohesion captures the influence of higher-order adjacencies on the cohomological features. We leverage our abstract structure theorems to obtain interval decompositions for the resulting geometric cohesion modules. Finally, we introduce biparameter persistence constructions for sheaf resilience, tracking whether global sections and cohomological obstructions remain detectable on thick or cohesive substructures during topological degradation.
\end{abstract}

\maketitle
%{\small \tableofcontents}

\section{Introduction}

One of the central features of one-parameter persistent homology is the existence of barcodes~\cite{Carlsson_Zomorodian05}. Algebraically, pointwise finite-dimensional persistence modules indexed by a totally ordered set decompose uniquely into interval modules~\cite{CrawleyBoevey15}; geometrically, this makes it possible to summarize the evolution of homological features along a filtration by a multiset of intervals. This situation changes drastically in multiparameter persistence. Modules indexed by non-totally ordered posets do not admit a comparable barcode classification in general. Already for multidimensional persistence, no complete discrete invariant analogous to the one-parameter barcode can be expected~\cite{Carlsson_Zomorodian09}. Although Krull--Remak--Schmidt type decompositions into indecomposables remain available in broad categorical settings~\cite{Bakke_Crawley20}, the indecomposable summands are typically too complicated to provide a usable barcode theory.

This obstruction has motivated several complementary approaches to multipersistence. Some works identify special classes of modules, such as rectangle-decomposable or block-decomposable modules, for which interval-type decompositions and computable invariants can be recovered~\cite{Bjerkevik21}, \cite{Botnan_Lebovici_Oudot22}, \cite{Bakke_Crawley20}, \cite{Lebovici_Lerch_Oudot24}. Other approaches replace exact barcodes by signed, rank-theoretic, or approximate decompositions~\cite{Bjerkevik25}, \cite{Botnan_Oppermann_Oudot25}, \cite{Botnan_Oppermann_Oudot_Scoccola24}. These developments show that, although multiparameter persistence is wild in general, structured families of modules may still admit meaningful interval descriptions.

The first aim of this paper is to prove structure theorems for one such family. More precisely, we study persistence modules over $\cP_N^{*,\op}$, the opposite of the poset of non-empty subsets of $[N]=\{0,\dots,N\}$. Since this indexing poset is not totally ordered, these modules lie genuinely outside the one-parameter setting. Theorems~\ref{Thm:Decomposition2} and~\ref{Thm:DecompositionN} give sufficient conditions under which they nevertheless decompose as direct sums of interval modules. In the general case, the decomposition is not only existential: Corollary~\ref{Cor:DecompositionN_Explicit} gives the interval summands and their multiplicities explicitly in terms of the dimensions of the module at a finite set of indices. Thus the structure theorem provides an effective way of computing the barcode-type data associated with this family of modules. 

These results are not intended as a classification theorem for arbitrary multipersistence modules. Rather, they identify a structured family inside a wild representation-theoretic setting. The hypotheses are expressed in terms of the transition morphisms of the diagram: some of them are required to be injective, and others to be isomorphisms. While such assumptions may look restrictive if read purely algebraically, we show that they are naturally realized by persistence modules arising from cellular sheaf cohomology. Cellular sheaves provide a flexible language for local algebraic data over combinatorial spaces: a sheaf on a simplicial complex assigns a vector space of local states to each simplex and linear maps describing how these states are transmitted, compared, or restricted along face relations. This viewpoint has become important in topological data analysis, network science, signal processing or network dynamics~\cite{Curry14}, \cite{Hansen20,Ghrist_Hansen19,Ghrist_Hansen21}, \cite{Shepard85}. 

The second aim of the paper is to show that the multipersistence modules studied here have a natural geometric origin in a theory of robustness for sheaf data over simplicial complexes. In this setting, robustness is not only a property of the underlying simplicial support; it also concerns the local algebraic data carried by the sheaf, the compatibility constraints imposed by its restriction morphisms, and the cohomology classes that measure global consistency or obstruction. In complex networks, robustness is usually understood as the ability of a system to preserve relevant structural or functional properties under failures, perturbations, or targeted attacks~\cite{Artime_Otros2024}, \cite{Barabasi_Posfai16}. Graph-based models typically represent such degradation through the deletion of vertices or edges. In higher-order network models, perturbations may also remove higher-dimensional simplices, thereby affecting not only connectivity and homology, but also the collective interactions supporting global structure~\cite{Bianconi_Ziff18}, \cite{Bobrowski_Skraba20}, \cite{Bobrowski_Skraba22}, \cite{Chen_Otros23}, \cite{Zhao_otros22_2}, \cite{Zhao_Otros22}. 

Classical Betti numbers detect connected components and holes, but they do not record how cohomology classes are supported by the combinatorics of the complex. In~\cite{Pablo_Dani_Dario_25}, thick and cohesion Betti numbers were introduced to refine this information, providing a measure of robustness. Thickness quantifies the dependence of cohomology on sufficiently high-dimensional support, whereas cohesion captures the strength of the higher-order adjacencies supporting these cohomology classes. Here we extend these constructions from constant coefficients to cellular sheaves. Sheaf thickness is obtained by restricting the sheaf to the coskeleta of the simplicial complex, giving a one-parameter persistence module that measures how sheaf cohomology changes when low-dimensional support is removed. Sheaf cohesion is obtained by restricting the associated sheaf on the face poset to subposets determined by selected dimensions. Since these dimension-selected subposets are not generally face posets of simplicial subcomplexes, we work with finite topological spaces and sheaves on posets~\cite{Barmak11}, \cite{Carmona_Fernando_Maestro_Juanfran20}, \cite{Curry14}, \cite{Fernando_17}, \cite{Shepard85}.  In degree zero, this construction leads to higher-order spaces of sections. These generalize ordinary global sections by measuring compatibility between simplices of prescribed dimensions without requiring compatibility through all intermediate strata. For instance, in a communication model, one may study agreement between vertex data and triangular interactions without forcing pairwise agreement along the edges. Thus, sheaf cohesion captures forms of higher-order compatibility that are invisible to ordinary global sections.

As the selected set of dimensions varies, the resulting cohesion sheaf cohomology groups form multipersistence modules over $\cP_N^{*,\op}$, precisely the finite posets appearing in the abstract structure theorems. Consequently, the abstract structure theorems have concrete consequences for sheaf cohesion. 
Theorems~\ref{Thm:DescomposicionCohesiveDim2} and~\ref{Thm:DescomposicionPN} show that cohesion persistence modules of cellular sheaves are interval-decomposable in several relevant geometric situations. In dimension $2$, this holds degreewise: $H^{0,\bullet}(X;\cF)$ decomposes under a local injectivity assumption, while $H^{1,\bullet}(X;\cF)$ and $H^{2,\bullet}(X;\cF)$ decompose for every cellular sheaf. In higher dimensions, the zero-th cohesion module decomposes for componentwise pure simplicial complexes with injective restriction morphisms. The purity assumption is natural in higher-order network models generated by facets of a fixed dimension or by interactions of a prescribed order~\cite{Bianconi_Rahmede16}, \cite{Courtney_Bianconi16}, \cite{Courtney_Bianconi17}.

Finally, we integrate these constructions with degradation processes on the simplicial support. This leads to biparameter persistence modules in which one parameter records the attack or failure process, while the other records either a thickness threshold or a selected set of dimensions. The associated ladder diagrams separate the persistence of the original sheaf cohomology from the persistence of the classes that remain visible in thick or cohesive parts of the complex. Their kernels, images, and cokernels distinguish classes that survive the restriction, classes that are lost, and classes that are created by passing to the selected support. In this way, the resilience of the simplicial support and that of the algebraic information encoded by the sheaf can be studied jointly.

\section{Preliminaries}\label{Sec:Prelim}

Throughout, $\Bbbk$ denotes a field, all simplicial complexes are finite, and all vector spaces are finite-dimensional over $\Bbbk$. 

\subsection{Cellular sheaves on simplicial complexes}
\label{Sec:CellularSheaves}We begin by recalling the standard framework of cellular sheaves on simplicial complexes. For a comprehensive background, we refer the reader to~\cite{Curry14} and~\cite{Hansen20}.

Let $X$ be a simplicial complex. We denote by $S^n(X)$ its set of $n$-simplices and by $(P_X,\face)$ its face poset ordered by inclusion. A cellular sheaf on $X$ is a functor $\mathcal{F}\colon P_X\to \vect$. Explicitly, $\mathcal{F}$ assigns a vector space $\mathcal{F}(\sigma)$ to every simplex $\sigma\in X$, and a linear map $\mathcal{F}_{\sigma\face\tau}\colon \mathcal{F}(\sigma)\to \mathcal{F}(\tau)$ to every inclusion $\sigma\face\tau$, such that $\mathcal{F}_{\sigma\face\sigma}$ is the identity map and $\mathcal{F}_{\tau\face\mu} \circ \mathcal{F}_{\sigma\face\tau} = \mathcal{F}_{\sigma\face\mu}$ whenever $\sigma\face\tau\face\mu$.

The space of global sections of $\mathcal{F}$, denoted by $\Gamma(X;\mathcal{F})$, is the vector space of local data assigned to vertices that are mutually compatible along all simplices of $X$: 
\[
\Gamma(X;\mathcal{F}) = \big\{ (x_v)_{v \in S^0(X)} \in \prod_{v \in S^0(X)} \mathcal{F}(v) \; : \; \mathcal{F}_{u\face \sigma}(x_u) = \mathcal{F}_{v\face \sigma}(x_v) \quad \forall \, u,v\face \sigma \in X \big\}\,.
\]

The cellular cochain spaces of $\mathcal{F}$ are defined as:
\[
C^n(X;\mathcal{F}) \coloneqq \prod_{\sigma\in S^n(X)}\mathcal{F}(\sigma)
\]
To define the coboundary maps $\delta^n_\mathcal{F}\colon C^n(X;\mathcal{F})\to C^{n+1}(X;\mathcal{F})$, we fix a total ordering on the vertices of $X$. For an $n$-cochain $x$, its coboundary is given by $(\delta^n_\mathcal{F} x)_\tau \coloneqq \sum_{j=0}^{n+1} (-1)^j \mathcal{F}_{\tau_j\face\tau}(x_{\tau_j})$,
where $\tau_j$ is the face of $\tau$ obtained by deleting its $j$-th vertex according to the fixed total order. The $n$-th cohomology space of $\mathcal{F}$ is $H^n(X;\mathcal{F}) \coloneqq \ker\delta^n_\mathcal{F}/\text{im }\delta^{n-1}_\mathcal{F}$. In degree zero, the cohomology space naturally identifies with the space of global sections of the cellular sheaf: $H^0(X;\mathcal{F})\simeq \Gamma(X;\mathcal{F})$. Furthermore, note that the cohomology of the constant cellular sheaf $\Bbbk$ coincides with the simplicial cohomology of $X$ with coefficients in $\Bbbk$.

If $f\colon X\to Y$ is a simplicial map and $\cF$ is a cellular sheaf on $Y$, its inverse image is the cellular sheaf $f^{-1}\cF$ on $X$ defined by
\[
f^{-1}\cF(\sigma)\coloneqq \cF(f(\sigma))\,,\qquad 
     f^{-1}\cF_{\sigma\face \tau}\coloneqq \cF_{f(\sigma)\face f(\tau)}\,.
\]
In particular, if $i\colon X'\hookrightarrow X$ is the inclusion of a subcomplex, we write $\cF_{|_{X'}}\coloneqq i^{-1}\cF$.

\begin{proposition}\label{Prop:FunctCellSheafCohomology}
For every simplicial map $f\colon X\to Y$, every cellular sheaf $\cF$ on $Y$, and every $n\geq 0$, there is a natural linear map
\[
    H^n(f)\colon H^n(Y;\cF)\longrightarrow H^n(X;f^{-1}\cF),
\]
which is functorial with respect to the composition of simplicial maps.
\end{proposition}
\begin{proof}
For each $n \ge 0$, we define the linear map $f^n \colon C^n(Y;\cF) \to C^n(X;f^{-1}\cF)$ by setting, for each ordered simplex $\sigma=(v_0,\dots,v_n)$, 
\[
    (f^ny)_\sigma\coloneqq \begin{cases}
        \sign(\epsilon_\sigma)\,y_{f(\sigma)} & \text{ if } \dim f(\sigma)=n\,,\\
        0 & \text{ otherwise,}
    \end{cases}
    \]
    where $\epsilon_\sigma$ is the permutation such that $f(v_{\epsilon_\sigma(0)}) < \dots < f(v_{\epsilon_\sigma(n)})$ according to the fixed ordering of $Y$. These maps commute with the coboundary operators ($\delta^n_{f^{-1}\cF} \circ f^n = f^{n+1} \circ \delta^n_\cF$). Therefore, they induce a linear map $H^n(f)$ on the cohomology spaces. The functorial properties $H^n(\Id) = \Id$ and $H^n(f \circ g) = H^n(g) \circ H^n(f)$ follow immediately from the definition.
\end{proof}

\subsection{Sheaves on finite posets}
\label{Sec:HacesPosets}
Every finite poset $P$ can be endowed with the Alexandrov topology, whose open sets are the upwards-closed subsets of $P$. For any $p\in P$, the smallest open set containing $p$ is $U_p\coloneqq \{q\in P \mid p\leq q\}$. This construction establishes an equivalence between finite posets and finite $T_0$ topological spaces~\cite{Barmak11}. 

A functor $\cF\colon P\to\vect$ determines a sheaf $\widehat{\cF}$ on $P$ defined by
\[
    \widehat{\cF}(U) \coloneqq \varprojlim_{p\in U}\cF(p)
\]
for every open set $U\subseteq P$. Conversely, a sheaf $\widehat{\cG}$ on $P$ determines a functor by assigning the vector space $\widehat{\cG}(U_p)$ to each $p \in P$, and the restriction map $\widehat{\cG}(U_p) \to \widehat{\cG}(U_q)$ induced by the inclusion $U_q \subseteq U_p$ to each relation $p \leq q$. These constructions define an equivalence of categories~\cite[Theorem~4.2.10]{Curry14}: 
\[
    \Fun(P,\vect)\simeq \Sh(P;\vect).
\]
In particular, when $P$ is the face poset of a simplicial complex $X$, this equivalence identifies cellular sheaves on $X$ with topological sheaves on the finite Alexandrov space $P_X$.

The explicit computation of the sheaf cohomology of $\widehat{\cF}$ can be carried out using its standard resolution (also known as the Roos resolution)~\cite{Fernando_17}. This yields a cochain complex with spaces $C^n(P;\widehat{\mathcal{F}}) \coloneqq \prod_{p_0<\dots<p_n}\mathcal{F}(p_n)$ and coboundary maps 
\begin{equation}\label{Eq:CoboundaryStandardResolution}
    (\delta^n x)_{p_0<\dots<p_{n+1}} \coloneqq \sum_{j=0}^{n} (-1)^j x_{p_0<\dots<\widehat{p_j}<\dots<p_{n+1}} + (-1)^{n+1} \mathcal{F}_{p_n\leq p_{n+1}} (x_{p_0<\dots<p_n})
\end{equation}
so that $H^n(P;\widehat{\cF})\simeq \ker \delta^n/\img\delta^{n-1}$~\cite[Theorem 2.15]{Fernando_17}.

Furthermore, sheaf cohomology is functorial with respect to order-preserving maps~\cite[Section 4.16]{Godement58}: any such map $f\colon P\to Q$ and sheaf $\widehat{\mathcal{G}}$ on $Q$ induce a linear map
\[
H^n(f)\colon H^n(Q;\widehat{\mathcal{G}}) \longrightarrow H^n(P;f^{-1}\widehat{\mathcal{G}})\,,\]
satisfying $H^n(\text{Id}) = \text{Id}$ and $H^n(f\circ g) = H^n(g)\circ H^n(f)$.

\begin{theorem}[{\!\!\cite[Theorem~1.4.2]{Shepard85}}] \label{Thm:EquivalenciaCohomologiacelularyhaces}
    Let $\cF$ be a cellular sheaf on a simplicial complex $X$. Then, for every $n\geq 0$, there is a natural isomorphism
    \[
        H^n(X;\cF) \simeq H^n(P_X;\widehat{\cF}).
    \]
\end{theorem}

\subsection{Thickness and cohesion in simplicial complexes}
\label{s:pre1}

We now recall the topological invariants introduced in~\cite{Pablo_Dani_Dario_25}. They refine the classical Betti numbers by incorporating information about the dimensions and higher-order adjacencies of the simplices supporting cohomology classes.

\paragraph{{\bf Thickness.}}
Let $X$ be a simplicial complex and let $q\in\bbN$. The $q$-coskeleton of $X$ is the simplicial subcomplex
\[
    X^q \coloneqq \{\sigma\in X : \exists\,\tau\in X \text{ such that } \sigma\face\tau \text{ and }\dim\tau\geq q\}.
\]
Thus, $X^1$ is obtained by removing isolated vertices, $X^2$ by removing vertices and edges not contained in any triangle, and so on. The $(n,q)$-th thick Betti number is defined as the dimension of the $n$-th cohomology space of the $q$-coskeleton:
\[
    \beta^{n,q}(X;\Bbbk) \coloneqq \dim H^n(X^q;\Bbbk).
\]
Since $X^0=X$, one recovers the classical Betti numbers when $q=0$, that is, $\beta^{n,0}(X;\Bbbk)=\beta^n(X;\Bbbk)$. 

If $N=\dim X$, the family of coskeleta forms a cofiltration of simplicial complexes
\[
    X=X^0\supseteq X^1\supseteq\cdots\supseteq X^N\supseteq X^{N+1}=\emptyset.
\]
This cofiltration induces a persistence module in cohomology,
\[
    H^n(X^\bullet;\Bbbk)\colon H^n(X^0;\Bbbk) \longrightarrow H^n(X^1;\Bbbk) \longrightarrow \cdots \longrightarrow H^n(X^N;\Bbbk),
\]
which records the evolution of cohomology classes as low-dimensional support is progressively removed. Intuitively, a class that persists to a larger value of $q$ is supported by higher-dimensional simplices and is, therefore, thicker.

\paragraph{{\bf Cohesion.}}
Although thickness captures the dimensions of the supporting simplices, it does not fully reflect the strength with which these simplices are attached to one another. To analyze this second feature, one must consider the selective deletion of simplices of prescribed dimensions. Since the resulting structure is generally no longer a simplicial complex, it is more natural to work with its face poset. Let $X$ be a simplicial complex of dimension $N$ and let
\[
    \h=\{h_0,\dots,h_m\}\subseteq [N]=\{0,1,\dots,N\}\,,
    \qquad
    0\leq h_0<\cdots<h_m\leq N.
\]
The $\h$-face poset of $X$ is the subposet $P_X^\h \coloneqq \{\sigma\in P_X : \dim\sigma\in\h\}$. Its order complex $\mathcal{K}(P_X^\mathbf{h})$ is a simplicial subcomplex of $\mathcal{K}(X)$, the barycentric subdivision of $X$. It is obtained by removing from $\mathcal{K}(X)$ the open stars of the vertices associated with simplices whose dimensions are not in $\mathbf{h}$:
\[
\mathcal{K}(P_X^\mathbf{h}) = \mathcal{K}(X) \smallsetminus \bigcup_{\substack{\dim\sigma\notin\mathbf{h}\\ \sigma \in X}} \st_{\mathcal{K}(X)}(\sigma)\,,
\]
where $\st_{\mathcal{K}(X)}(\sigma) \coloneqq \{\tau \in \mathcal{K}(X) : \sigma \in \tau\}$. The $(n,\h)$-th cohesive Betti number is
\begin{equation}\label{Def:CohesiveBetti}
    \beta^{n,\h}(X;\Bbbk)
    \coloneqq
    \dim H^n(\mathcal{K}(P_X^\h);\Bbbk)\,.
\end{equation}
For $\h=[N]$, one has $P_X^{[N]}=P_X$ and $\mathcal{K}(P_X)=\mathcal{K}(X)$, recovering the classical Betti numbers: $\beta^{n,[N]}(X;\Bbbk)= \beta^n(X;\Bbbk)$.

The collection of $\h$-face posets defines a filtration indexed by $\cP^*_N$, the poset of non-empty subsets of $[N]$ ordered by inclusion. Taking cohomology yields a multiparameter persistence module indexed by the opposite poset of $\cP^*_N$:
\begin{equation}\label{Eq:CohesionModuleConstantSheaf}
        H^n(P_X^\bullet;\Bbbk)\colon
    \cP^{*,\op}_N
    \longrightarrow
    \vect.
\end{equation}
For each $\h\in\cP^*_N$, the inclusion $P_X^\h\subseteq P_X$ induces a linear map
\[
    \varphi^{n,\h}\colon
    H^n(X;\Bbbk)
    \simeq
    H^n(P_X;\Bbbk)
    \longrightarrow
    H^n(P_X^\h;\Bbbk).
\]
Its image measures the cohomology classes of $X$ that remain detectable after restricting to the dimensions in $\h$:
\[
    \beta_{\img}^{n,\h}(X;\Bbbk)
    \coloneqq
    \dim\img\varphi^{n,\h}.
\]
The kernel and cokernel of $\varphi^{n,\h}$ respectively record cohomology classes destroyed and created by this selective deletion.

\subsection{Persistence}
Let $P$ be a poset. A pointwise finite-dimensional (p.f.d.) persistence module over $P$ is a functor $\cM\colon P\to \vect$. Given two persistence modules $\cM, \cN\colon P\to \vect$, their direct sum $\cM\oplus\cN$ is the persistence module defined pointwise by $(\cM\oplus\cN)_p\coloneqq \cM_p\oplus\cN_p$ and $(\cM\oplus\cN)_{p\leq q}\coloneqq \cM_{p\leq q}\oplus\cN_{p\leq q}$. A persistence module is called indecomposable if it cannot be written as a direct sum of non-trivial submodules. 

In the pointwise finite-dimensional setting, every persistence module decomposes uniquely (up to reordering) as a direct sum of indecomposable modules with local endomorphism rings~\cite{Bakke_Crawley20}. Consequently, the classification of p.f.d. persistence modules reduces to characterizing these indecomposables. A central class of such modules is given by interval modules, which are defined via the combinatorial structure of the poset.

\begin{definition}
A subset $I \subseteq P$ is an interval if it is both convex and connected. Explicitly:
\begin{itemize}
    \item $I$ is convex if for all $p,q \in I$ and $r \in P$ such that $p \leq r \leq q$, it holds that $r \in I$.
    \item $I$ is connected if for all $p,q \in I$ there exists a sequence $\{r_i\}_{i=0}^n \subseteq I$ such that $r_0=p$, $r_n=q$, and either $r_i \leq r_{i+1}$ or $r_i \geq r_{i+1}$ for each $0\leq i < n$.
\end{itemize}
\end{definition}

Given an interval $I \subseteq P$, the interval persistence module $\Bbbk[I]$ is defined as
\begin{equation*}
\Bbbk[I]_p=\begin{cases}\Bbbk & \text{ if } p\in I ,\\
0 & \text{ otherwise,}
\end{cases}
\qquad \text{ and } \qquad
\Bbbk[I]_{p\leq q}=\begin{cases}
\mathrm{Id}_\Bbbk & \text{ if } p,q\in I, \\
0 & \text{ otherwise.}
\end{cases}
\end{equation*}
\begin{proposition}[{\!\cite[Proposition 2.2]{Botnan_Lesnick18}}]
    For any interval $I$, the module $\Bbbk[I]$ is indecomposable and its endomorphism ring $\End(\Bbbk[I])$ is local.
\end{proposition}

When the indexing set is totally ordered, or in the case of zigzag persistence, the classification is entirely determined by these interval modules:

\begin{theorem}[Structure Theorem,{~\cite[Theorem 1.2]{Bakke_Crawley20}}]\label{Thm:Descomposicion}
Every persistence module $\cM\colon T\to \vect$ indexed by a totally ordered set $T$ decomposes as a direct sum of interval modules.
\end{theorem}

\begin{theorem}[{Gabriel's Theorem,~\cite[Theorem 2.5]{Carlsson_Silva10}}]\label{Thm:Gabriel}
Every pointwise finite-dimensional zigzag persistence module decomposes as a direct sum of interval modules.
\end{theorem}

\subsection{Persistent cohomology of sheaves}\label{Sec:Cohomologiapersistentehaces}

Given the central role that persistent sheaf cohomology of topological type will play in this paper, we now summarize the foundations introduced by Russold~\cite{Russold22}. 

Consider a cofiltration of simplicial complexes of the form
\[
X^{\bullet}\colon \,X^0 \supseteq X^1 \supseteq X^2 \supseteq \dots \supseteq X^M\,,
\]
and a cellular sheaf $\cF$ on the initial complex $X^0$. The successive restriction of $\cF$ to the different subcomplexes $X^j$ defines a collection of cellular sheaves $\{\cF_{|_{X^j}}\}_{j=0}^M$. Moreover, for each $1\leq j\leq M$, one has the relation  $\cF_{|_{X^j}}=i_j^{-1}(\cF_{|_{X^{j-1}}})$, where $i_j\colon X^{j}\hookrightarrow X^{j-1}$ is the inclusion map. By the functoriality of cellular sheaf cohomology, we obtain a sequence of linear maps:
\begin{equation*}
H^n(X^0;\cF)\longrightarrow H^n(X^1;\cF_{|_{X^1}})\longrightarrow H^n(X^2;\cF_{|_{X^2}})\longrightarrow \dots \longrightarrow H^n(X^M;\cF_{|_{X^M}})\,.
\end{equation*}
This yields a persistence module 
\[
H^{n}(X^\bullet;\cF_{|_{X^\bullet}})\colon [M]\to \vect\,,
\]
which encodes the changes undergone by the algebraic information of the sheaf $\cF$ along the cofiltration $X^\bullet$.

As Russold points out, this construction naturally extends to the case where the initial cofiltration $X^\bullet$ is indexed by a poset $P$ with a minimum element $p_0$, and where $\cF$ is a cellular sheaf on the simplicial complex $X^{p_0}$ associated with that minimum. The restrictions of $\cF$ to the remaining complexes of $X^\bullet$ and the linear maps induced in cohomology define a persistence module 
\[
H^{n}(X^\bullet;\cF_{|_{X^\bullet}})\colon P\to \vect\,.
\]
The persistence modules resulting from this procedure are called persistence modules of topological type or persistence modules of type T~\cite{Russold22}.

\section{Structure theorems}\label{Sec:StructureTheorems}

In this section, we establish structure theorems for persistence modules of the form $$\cM\colon\mathcal{P}^{*,\text{op}}_N\longrightarrow \vect\,.$$
As a functor, $\cM$ assigns a linear map $\cM_{\h \leq \h'}\colon \cM_{\h'} \longrightarrow \cM_{\h}$ to each inclusion $\h \leq \h'$ in $\mathcal{P}^{*}_N$. We prove that, provided these linear maps satisfy specific injectivity and isomorphism conditions, the module decomposes uniquely into a direct sum of interval persistence modules.

\begin{theorem}\label{Thm:Decomposition2}
    Let $\cM\colon \cP^{*,\op}_{2}\longrightarrow \vect$ be a p.f.d. persistence module. Suppose there exists an index $i\in [2]$ satisfying the following conditions:
    \begin{enumerate}
        \item The linear map $\cM_{[2]\smallsetminus \{i\}\leq [2]}$ is an isomorphism.
        \item \label{cond:injectivity} The linear map $\cM_{\h\leq\h'}$ is injective for all $\h\neq \{i\}$ and $\h' \geq \h$.
    \end{enumerate}
    Then $\cM$ decomposes as a direct sum of interval persistence modules.
\end{theorem}
\begin{proof} For notational simplicity, we present the proof for the case $i=2$. The arguments for the other cases are completely analogous.

    Let $x_0\in \cM_{\{0,1,2\}}$ be a nonzero vector, and define the set
    \[
    I_0=\{\h\in \cP^{*,\op}_2 : \cM_{\h\leq \{0,1,2\}}(x_0)\neq 0\}\, .
    \]
    From the injectivity required by condition~\eqref{cond:injectivity}, it follows that either $I_0=\cP^{*,\op}_2\smallsetminus\{2\}$ or $I_0=\cP^{*,\op}_2$.

    Let $\cN_{\{2\}}$ be a complement of $\langle \cM_{\{2\}\leq\{0,1,2\}}(x_0)\rangle$, and let $\{y_1,\dots,y_l\}$ be a basis of $\cN_{\{2\}}$. We can find a basis $\{x_0,x_1,\dots,x_k\}$ of $\cM_{\{0,1,2\}}$ such that $\cM_{\{2\}\leq\{0,1,2\}}(x_i)\in \cN_{\{2\}}$ for all $1\leq i\leq k$. To do this, first extend $\{x_0\}$ to a basis $\{x_0,x_1',\dots,x_k'\}$. If $\cM_{\{2\}\leq\{0,1,2\}}(x_0)$ is nonzero, then for each $1\leq j\leq k$ we can write
    \[
    \cM_{\{2\}\leq\{0,1,2\}}(x_j')=\lambda_0^j\, \cM_{\{2\}\leq\{0,1,2\}}(x_0)+\lambda_1^j \,y_1+\cdots+\lambda_l^j\,y_l
    \]
    for some scalars $\lambda_0^j, \dots, \lambda_l^j$. By setting $x_j\coloneqq x_j'-\lambda_0^j\,x_0$, we obtain a basis $\{x_0,x_1,\dots,x_k\}$ of $\cM_{\{0,1,2\}}$ satisfying the desired condition.

    Denote by $\cN_{\{0,1,2\}}$ the subspace spanned by $\{x_1,\dots,x_k\}$ and define
    \[
    \begin{aligned}
        \cN_{\{0,1\}} &\coloneqq \cM_{\{0,1\}\leq\{0,1,2\}}(\cN_{\{0,1,2\}})\,, \\[1.5ex]
        \cN_{\{0,2\}}^{\prime} &\coloneqq \cM_{\{0,2\}\leq\{0,1,2\}}(\cN_{\{0,1,2\}})\,, \\[1.5ex]
        \cN_{\{1,2\}}^{\prime} &\coloneqq \cM_{\{1,2\}\leq\{0,1,2\}}(\cN_{\{0,1,2\}})\,.
    \end{aligned}
    \]
    Next, let $\cN_{\{0,2\}}^{\prime\prime}$ and $\cN_{\{1,2\}}^{\prime\prime}$ be subspaces such that
    \[
    \begin{aligned}
        \cM_{\{0,2\}}& =\langle\cM_{\{0,2\}\leq\{0,1,2\}}(x_0)\rangle \oplus \cN_{\{0,2\}}^{\prime}\oplus \cN_{\{0,2\}}^{\prime\prime}\, ,\\[1.5ex]
        \cM_{\{1,2\}}&=\langle \cM_{\{1,2\}\leq\{0,1,2\}}(x_0)\rangle \oplus \cN_{\{1,2\}}^{\prime}\oplus \cN_{\{1,2\}}^{\prime\prime}\, .
    \end{aligned}
    \]
    Repeating the preceding argument, we may further assume that
    \[
    \begin{aligned}
        \cM_{\{2\}\leq\{0,2\}}(\cN_{\{0,2\}}^{\prime\prime})&\subseteq \cN_{\{2\}}\,,\\[1.5ex]
        \cM_{\{2\}\leq\{1,2\}}(\cN_{\{1,2\}}^{\prime\prime})& \subseteq \cN_{\{2\}}\, .
    \end{aligned}
    \]
    Thus, the subspaces
    \[
    \begin{aligned}
    \cN_{\{0,2\}}&\coloneqq \cN_{\{0,2\}}^{\prime}\oplus \cN_{\{0,2\}}^{\prime\prime}\, ,\\[1.5ex]
    \cN_{\{1,2\}}&\coloneqq \cN_{\{1,2\}}^{\prime}\oplus \cN_{\{1,2\}}^{\prime\prime}\, ,
    \end{aligned}
    \]
    satisfy
    \[
    \begin{aligned}
        \cM_{\{2\}\leq\{0,2\}}(\cN_{\{0,2\}})&\subseteq \cN_{\{2\}}\,,\\[1.5ex]
        \cM_{\{2\}\leq\{1,2\}}(\cN_{\{1,2\}})& \subseteq \cN_{\{2\}}\, .
    \end{aligned}
    \]
    For $\h=\{0\}$ and $\h=\{1\}$, define
    \[
    \begin{aligned}
        \cN_{\{0\}}^{\prime}&\coloneqq \cM_{\{0\}\leq\{0,2\}}(\cN_{\{0,2\}})\, ,\\[1.5ex]
        \cN_{\{1\}}^{\prime}&\coloneqq \cM_{\{1\}\leq\{1,2\}}(\cN_{\{1,2\}})\, .
    \end{aligned}
    \]
    We also choose subspaces $\cN_{\{0\}}^{\prime\prime}$ and $\cN_{\{1\}}^{\prime\prime}$ such that
    \[
    \begin{aligned}
        \cM_{\{0\}}&=\langle \cM_{\{0\}\leq\{0,1,2\}}(x_0)\rangle \oplus \cN_{\{0\}}^{\prime} \oplus \cN_{\{0\}}^{\prime\prime}\, ,\\[1.5ex]
        \cM_{\{1\}}&=\langle \cM_{\{1\}\leq\{0,1,2\}}(x_0)\rangle \oplus \cN_{\{1\}}^{\prime} \oplus \cN_{\{1\}}^{\prime\prime} \, .
    \end{aligned}
    \]
    Finally, define
    \[
    \begin{aligned}
        \cN_{\{0\}}&\coloneqq \cN_{\{0\}}^{\prime} \oplus \cN_{\{0\}}^{\prime\prime}\,,\\[1.5ex]
       \cN_{\{1\}}&\coloneqq \cN_{\{1\}}^{\prime} \oplus \cN_{\{1\}}^{\prime\prime}\, .
    \end{aligned}
    \]
    The subspaces $\cN_{\h}$ we have defined satisfy $\cM_{\h\leq\h'}(\cN_{\h'})\subseteq \cN_{\h}$ for all $\h\leq \h'$. Therefore, the subspaces $\cN_{\h}$ form a submodule $\cN$ of the persistence module $\cM$ such that
    \[
    \cM \simeq \Bbbk[I_0]\oplus\cN\, .
    \]
    Moreover, $\dim\cN_{\{0,1,2\}}=\dim\cN_{\{0,1\}}=\dim \cM_{\{0,1,2\}}-1$.

    Applying the same argument to the module $\cN$, we obtain a decomposition
    \[
    \cN\simeq \Bbbk[I_1]\oplus \cN_1\, .
    \]
    Repeating the argument recursively, we arrive at a decomposition
    \begin{equation}\label{Eq:MkDescomposicionDim2}
        \cM \simeq \Bbbk[I_0] \, \oplus\, \Bbbk[I_1]\,\oplus\,\cdots\,\oplus\, \Bbbk[I_k]\,\oplus \cN_k\, ,
    \end{equation}
    where $\cN_k$ has support in the interval $I=\{\{0\},\{1\},\{2\},\{0,2\},\{1,2\}\}$. This interval forms a zigzag, and by Gabriel's theorem (Theorem~\ref{Thm:Gabriel}), the existence of the decomposition of $\cM$ into interval modules follows. 
\end{proof}

\begin{example}\label{Ex:Decomposition2}
    Consider a persistence module $\cM\colon \cP^{*,\op}_2\to \mathbf{vect}_{\mathbb{R}}$ with the structural configuration depicted in Figure~\ref{Fig:Decomposition2-Example}, which indicates both the point-wise dimensions and the properties of the linear maps (injectivity and isomorphisms). That is, $\cM$ satisfies the conditions in Theorem~\ref{Thm:Decomposition2} for the index $i=2$.

\begin{figure}[htb!]
    \centering
    \begin{subfigure}{0.45\linewidth}
        \centering
        \includegraphics[height=3.5cm]{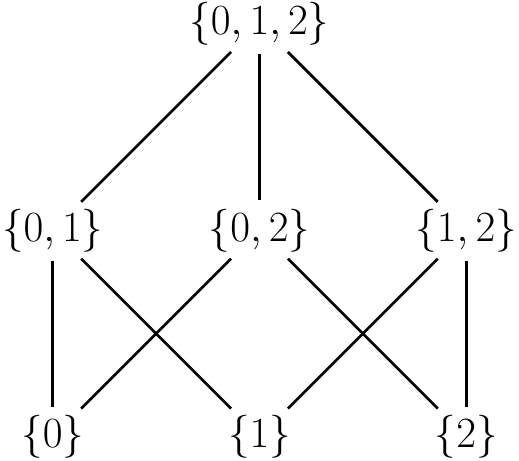}
        \caption{$\cP^{*}_2$}
    \end{subfigure}
    \hspace{1ex}
    \begin{subfigure}{0.45\linewidth}
        \centering
        \includegraphics[height=3.5cm]{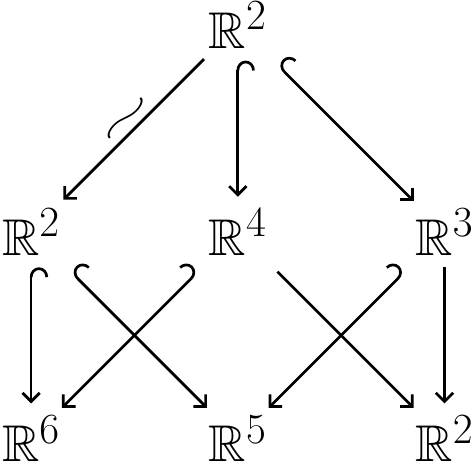}
        \caption{$\cM$}
    \end{subfigure}
    \caption{(A) Hasse diagram of the underlying poset $\cP^{*}_2$. (B) A persistence module over $\cP^{*,\op}_{2}$.}
    \label{Fig:Decomposition2-Example}
\end{figure}

    For each $\h \in \cP^{*,\op}_2$, let $\mathcal{B}_{\h}=\{e_i^{\h}\}_{i=1}^{\dim \cM_{\h}}$ be a basis of the vector space $\cM_{\h}$. Suppose that, with respect to our chosen bases, the matrices of the maps required for our computations are given by:
    \begin{align}\label{Eq:MapsDecomposition2}
        \cM_{\{0,2\} \leq\{0,1,2\}} &= \begin{pmatrix} 
            1 & 0 \\ 0 & 1 \\ 0 & 1 \\ 0 & 1 
        \end{pmatrix}, &
        \cM_{\{1,2\} \leq \{0,1,2\}} &= \begin{pmatrix} 
            1 & 1 \\ 0 & -1 \\ 0 & -1 
        \end{pmatrix}, \\[2ex]
        \cM_{\{2\} \leq \{0,2\}} &= \begin{pmatrix} 
            1 & 1 & 0 & 0 \\ 0 & 0 & 0 & 0 
        \end{pmatrix}, & 
        \cM_{\{2\} \leq \{1,2\}} &= \begin{pmatrix} 
            1 & 0 & 0 \\ 0 & 0 & 0 
        \end{pmatrix}.
    \end{align}
    By Theorem~\ref{Thm:Decomposition2}, we are guaranteed that $\cM$ decomposes as a direct sum of interval modules. We will now explicitly calculate this decomposition by following the constructive steps outlined in the proof of the theorem.

    First, notice that: 
    \begin{align*}
        \cM_{\{2\}\leq\{0,1,2\}}(e_1^{\{0,1,2\}}) &= \left(\cM_{\{2\}\leq\{1,2\}}\circ \cM_{\{1,2\}\leq\{0,1,2\}}\right)(e_1^{\{0,1,2\}}) \\[1.5ex] 
        &=\begin{pmatrix} 
            1 & 0 & 0 \\ 0 & 0 & 0 
        \end{pmatrix} \begin{pmatrix} 
            1 & 1 \\ 0 & -1 \\ 0 & -1 
        \end{pmatrix} \begin{pmatrix}
            1\\ 0 
        \end{pmatrix} = \begin{pmatrix}
            1 \\ 0 
        \end{pmatrix} = e_1^{\{2\}},
    \end{align*}
    and similarly,
    \begin{align*}
        \cM_{\{2\}\leq\{0,1,2\}}\big(e_2^{\{0,1,2\}}\big) &= \left(\cM_{\{2\}\leq\{1,2\}}\circ \cM_{\{1,2\}\leq\{0,1,2\}}\right)(e_2^{\{0,1,2\}}) \\[1.5ex] 
        &=\begin{pmatrix} 
            1 & 0 & 0 \\ 0 & 0 & 0 
        \end{pmatrix} \begin{pmatrix} 
            1 & 1 \\ 0 & -1 \\ 0 & -1 
        \end{pmatrix} \begin{pmatrix}
            0\\ 1 
        \end{pmatrix} = \begin{pmatrix}
            1 \\ 0 
        \end{pmatrix} = e_1^{\{2\}}.
    \end{align*}
    Therefore, we consider the following change of basis in $\cM_{\{0,1,2\}}$:
    \[
        \mathcal{B}_{\{0,1,2\}}\longrightarrow \mathcal{B}_{\{0,1,2\}}'\coloneqq \{e_1^{\{0,1,2\}},\,e_2^{\{0,1,2\}}-e_1^{\{0,1,2\}}\}.
    \]
    Since $\cM_{\{2\}\leq \{0,1,2\}}(e_{1}^{\{0,1,2\}})=e_{1}^{\{2\}}$ and $\cM_{\{2\}\leq \{0,1,2\}}(e_{2}^{\{0,1,2\}}-e_{1}^{\{0,1,2\}})=0$, it follows that $\cM$ has $\mathbb{R}[\cP^{*,\op}_2]$ and $\mathbb{R}[\cP^{*,\op}_2\smallsetminus\{2\}]$ as direct summands. To complete the decomposition, we perform the following changes of basis:
    \begin{align*}
        \mathcal{B}_{\{0,2\}}& \longrightarrow \mathcal{B}_{\{0,2\}}'\coloneqq \{e_{1}^{\{0,2\}},\,-e_1^{\{0,2\}}+e_{2}^{\{0,2\}}+e_{3}^{\{0,2\}}+e_4^{\{0,2\}},\, e_3^{\{0,2\}},\,e_4^{\{0,2\}}\},\\[1.5ex]
        \mathcal{B}_{\{1,2\}}& \longrightarrow \mathcal{B}_{\{1,2\}}'\coloneqq \{e_1^{\{1,2\}},\,-e_2^{\{1,2\}}-e_3^{\{1,2\}}, \,e_3^{\{1,2\}}\},
    \end{align*}
    where the first two vectors of $\mathcal{B}_{\{0,2\}}'$ and $\mathcal{B}_{\{1,2\}}'$ correspond to the images of the vectors in $\mathcal{B}_{\{0,1,2\}}'$ under their respective restriction maps. With respect to these new bases, the matrices become: 
    \begin{equation}
        \label{Eq:MatricesNewBases}
        \begin{aligned}
            \cM_{\{0,2\} \leq\{0,1,2\}} &= \begin{pmatrix} 
                1 & 0 \\ 0 & 1 \\ 0 & 0 \\ 0 & 0 
            \end{pmatrix}, &
            \cM_{\{1,2\} \leq \{0,1,2\}} &= \begin{pmatrix} 
                1 & 0 \\ 0 & 1 \\ 0 & 0 
            \end{pmatrix}, \\[2ex]
            \cM_{\{2\} \leq \{0,2\}} &= \begin{pmatrix} 
                1 & 0 & 0 & 0 \\ 0 & 0 & 0 & 0 
            \end{pmatrix}, & 
            \cM_{\{2\} \leq \{1,2\}} &= \begin{pmatrix} 
                1 & 0 & 0 \\ 0 & 0 & 0 
            \end{pmatrix}.
        \end{aligned}
    \end{equation}
    Let us define the subspaces $\cN_{\{0,2\}}\coloneqq \langle e_3^{\{0,2\}}, \,e_4^{\{0,2\}}\rangle$ and $\cN_{\{1,2\}}\coloneqq \langle e_3^{\{1,2\}}\rangle$. Observing the block structure of the matrices in Equation~\eqref{Eq:MatricesNewBases}, it is clear that:
    \begin{align}\label{Eq:TrivialMaps}
        \cM_{\{2\}\leq \{0,2\}}(\cN_{\{0,2\}})=\cM_{\{2\}\leq \{1,2\}}(\cN_{\{1,2\}})=0\,.
    \end{align}
    Consequently, $\cM$ decomposes as the direct sum:
    \[
        \cM\simeq \mathbb{R}[\cP^{*,\op}_2]\oplus \mathbb{R}[\cP^{*,\op}_2\smallsetminus\{2\}]\oplus \cN\,,
    \]
    where the submodule $\cN$ takes the following form: 
    \begin{center}
        \includegraphics[width=0.5\linewidth]{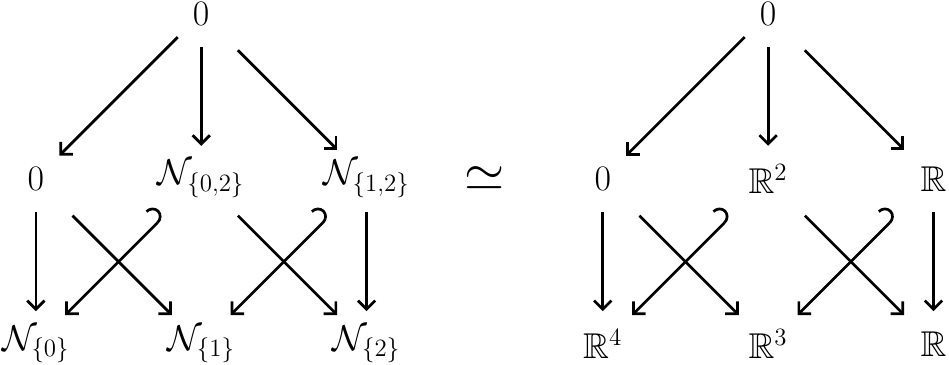}
    \end{center}
    Furthermore, since the restriction maps $\cN_{\{2\}\leq \{0,2\}}$ and $\cN_{\{2\}\leq \{1,2\}}$ are trivial (Equation~\eqref{Eq:TrivialMaps}), $\cN$ decomposes as:
    \begin{center}
        \includegraphics[width=0.95\linewidth]{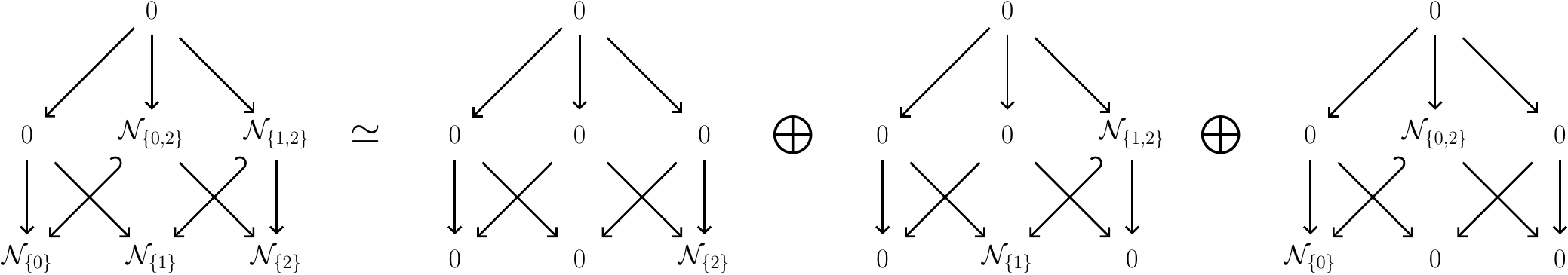}
    \end{center}
    Decomposing the summands of this final expression, we conclude that (Figure~\ref{Fig:Decomposition2-Barcode}):
    \begin{equation}\label{Eq:Decomposition2-Decomposition}
    \begin{split}
         \cM\simeq &\, \mathbb{R}[\cP^{*,\op}_2]\,\oplus \,\mathbb{R}[\cP^{*,\op}_2\smallsetminus\{2\}]\, \oplus\, \mathbb{R}[\{\{0\},\{0,2\}\}]^2\,\oplus\\[1ex]
         & \mathbb{R}[\{\{1\},\{1,2\}\}]\,\oplus \,\mathbb{R}[\{0\}]^2\oplus \mathbb{R}[\{1\}]^2 \,\oplus\, \mathbb{R}[\{2\}].
    \end{split}
    \end{equation}
\end{example}

\begin{figure}[htb!]
    \centering
    \includegraphics[width=0.8\linewidth]{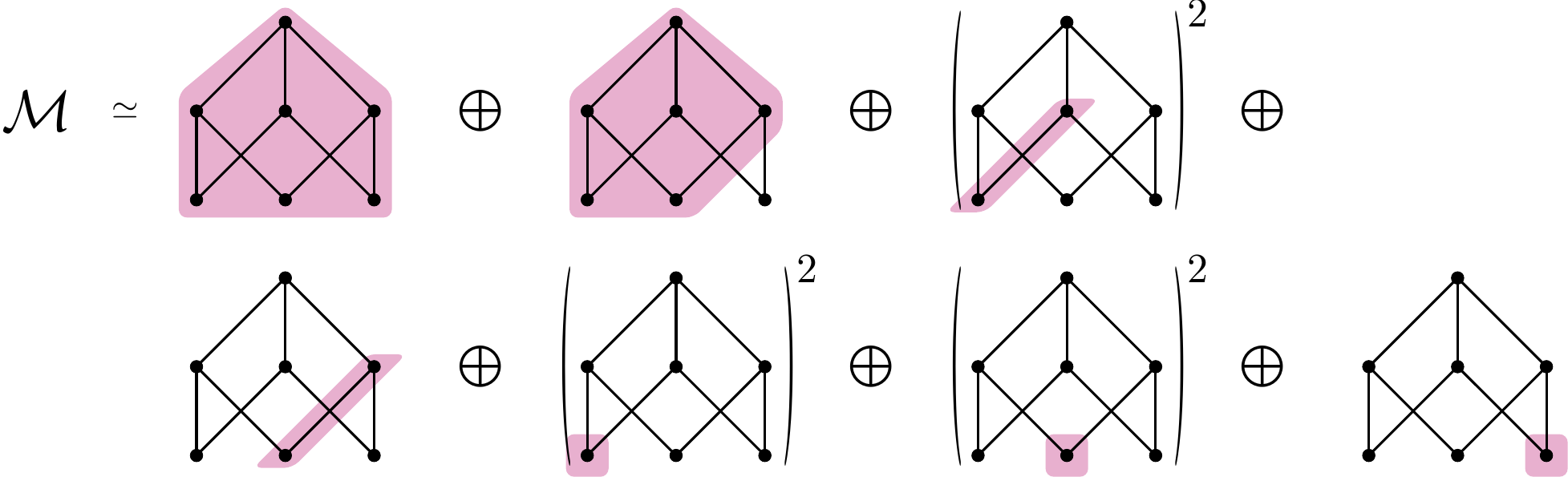}
    \caption{Interval decomposition of the persistence module $\cM$ in Example~\ref{Ex:Decomposition2}.}
    \label{Fig:Decomposition2-Barcode}
\end{figure}

\begin{lemma}\label{Lem:DecompositionMinimalInjectives}
    Let $P$ be a poset and let $\cM\colon P\to \vect$ be a p.f.d. persistence module such that for each $p\leq q \in P$ the morphism $\cM_{p\leq q}$ is injective. Suppose that $P$ has a minimum element $p_0$ and denote by $n_0$ the dimension of $\cM_{p_0}$.
    Then, $\cM$ has $\Bbbk[P]$ as a direct summand $n_0$ times, that is,
    \[
    \cM\simeq \Bbbk[P] \,\oplus\, \stackrel{n_0}{\cdots}\,\oplus\, \Bbbk[P] \, \oplus \, \cM'\, ,
    \]
    for some submodule $\cM'$ of $\cM$.
\end{lemma}
\begin{proof}
    Let $\cM''$ be the submodule of $\cM$ defined for each $p\in P$ by
    \[
    \cM_p^{\prime\prime}\coloneqq \img \cM_{p_0\leq p}\, .
    \]
    Due to the injectivity of the restriction maps, we can find a complement $\cM_p'$ for each $\cM_p''$ in such a way that $\cM_{p\leq q}(\cM_p') \subseteq \cM_q'$ for all $p\leq q$. This guarantees that $\cM$ splits as 
    \[
    \cM= \cM''\oplus \cM'\,.
    \]
    Finally, all morphisms of $\cM''$ are isomorphisms, and therefore
    \[
    \cM''\simeq \Bbbk[P] \,\oplus\, \stackrel{n_0}\cdots\,\oplus\, \Bbbk[P]\,.
    \]
\end{proof}

\begin{theorem}\label{Thm:DecompositionN}
    Let $\cM\colon\cP^{*,\op}_N\longrightarrow \vect$ be a persistence module. Suppose that the following conditions hold:
    \begin{enumerate}
        \item The linear map $\cM_{\h\leq \h'}$ is injective for all $\h\leq \h'$.
        \item For every $\h=\{h_0,\dots,h_m\}\in \cP^{*,\op}_N$ with $m\geq 2$ and each $1\leq i\leq m$, the linear map 
        \[
            \cM_{\h_i\leq\h}\colon \cM_{\{h_0,\dots,h_m\}}\longrightarrow \cM_{\{h_0,\dots,\widehat{h}_i,\dots,h_m\}}
        \]
        is an isomorphism, where $\h_i\coloneqq\{h_0,\dots,\widehat{h}_i,\dots,h_m\}$.
    \end{enumerate}
    Then $\cM$ decomposes as a direct sum of interval persistence modules.
\end{theorem}

\begin{proof}
   By Lemma~\ref{Lem:DecompositionMinimalInjectives}, the persistence module $\cM$ decomposes as
    \[
    \cM\simeq \Bbbk[\cP^{*,\op}_N]\,\oplus\,\stackrel{n_0}\cdots\,\oplus\,\Bbbk[\cP^{*,\op}_N]\,\oplus\,\cM'\,,
    \]
    where, following the notation of the lemma, $n_0\coloneqq\dim \cM_{[N]}$ and $\cM'$ is a submodule of $\cM$.
    
    Given that the maps $\cM_{\h_i\leq\h}$ are isomorphisms for every $\h$ and each $1 \leq i \leq m$, the support of $\cM'$ is contained in $\cP^*(\{1,\dots,N\})\cup\{0\}$. Since this union is disjoint, we can write
    \[
\cM'={\cM'}_{|_{\{0\}}}\,\oplus\,{\cM'}_{|_{\cP^*(\{1,\dots,N\})}}\,.
\]
On the one hand, letting $n_0'\coloneqq \dim \cM'_{\{0\}}=\dim \cM_{\{0\}}-\dim \cM_{[N]}$, we have
\[
{\cM'}_{|_{\{0\}}}\simeq \Bbbk[\{0\}]\,\oplus\,\stackrel{n_0'}\cdots\,\oplus\,\Bbbk[\{0\}]\, .
\]
On the other hand, the persistence module $\cM'_{|_{\cP^*(\{1,\dots,N\})}}$ satisfies the hypotheses of Lemma~\ref{Lem:DecompositionMinimalInjectives}, and therefore decomposes as
\[
{\cM'}_{|_{\cP^*(\{1,\dots,N\})}}\simeq \Bbbk[\cP^*(\{1,\dots,N\})]\,\oplus\,\stackrel{n_1}\cdots\,\oplus\, \Bbbk[\cP^*(\{1,\dots,N\})]\,\oplus\,\cM'' \, ,
\]
where the number of summands is given by
\[
n_1=\dim \cM'_{\{1,\dots,N\}}=\dim \cM_{\{1,\dots,N\}}-\dim \cM_{[N]}\, .
\]
By the same argument as before, $\cM''$ has its support contained in $\cP^*(\{2,\dots,N\})\cup \{1\}$. Proceeding recursively in this manner yields the decomposition.
\end{proof}
\begin{corollary}\label{Cor:DecompositionN_Explicit}
    Under the hypotheses of Theorem~\ref{Thm:DecompositionN}, the decomposition of $\cM$ is explicitly given by:
    \[
        \cM\simeq \bigoplus_{j=0}^N  \Bbbk[\cP^*(\{j,j+1,\dots,N\})]^{n_j}\oplus \bigoplus_{j=0}^{N-1}\Bbbk[\{j\}]^{n_j'}\,,
    \]
    where the multiplicities are:
    \begin{itemize}
        \item $n_0=\dim\cM_{[N]}$,
        \item $n_j=\dim \cM_{\{j,j+1,\dots,N\}}-\dim\cM_{\{j-1,j,\dots,N\}}$ for all $1\leq j\leq N$,
        \item $n_j'=\dim \cM_{\{j\}}-\dim\cM_{\{j,j+1,\dots,N\}}$ for all $0\leq j\leq N-1$.
    \end{itemize}
\end{corollary}

\begin{example}
    Consider any persistence module $\cM\colon \cP^{*,\op}_{3}\to \mathbf{vect}_\mathbb{R}$ having the structure depicted in Figure~\ref{Fig:DecompositionNExample}, where all transition maps are injective and the blue arrows denote isomorphisms. Then, by Corollary~\ref{Cor:DecompositionN_Explicit}, $\cM$ decomposes as:
    \[
    \cM\simeq \mathbb{R}[\cP_{3}^{*,\op}] \oplus \mathbb{R}[\cP^*(\{1,2,3\})]^2\oplus \mathbb{R}[\cP^*(\{2,3\})]\oplus \mathbb{R}[\{0\}]\oplus \mathbb{R}[\{1\}]^2\oplus \mathbb{R}[\{2\}]^3\oplus\mathbb{R}[\{3\}]\,.
    \]
    This interval decomposition is illustrated in Figure~\ref{Fig:DecompositionNBarcode}. It is worth noting that for this particular class of modules, the decomposition does not depend on the specific choice of the linear maps, but solely on the dimensions of the vector spaces at each node.
\end{example}
\begin{figure}[htb!]
    \centering
    \begin{subfigure}{0.45\linewidth}
        \centering
        \includegraphics[width=0.9\linewidth]{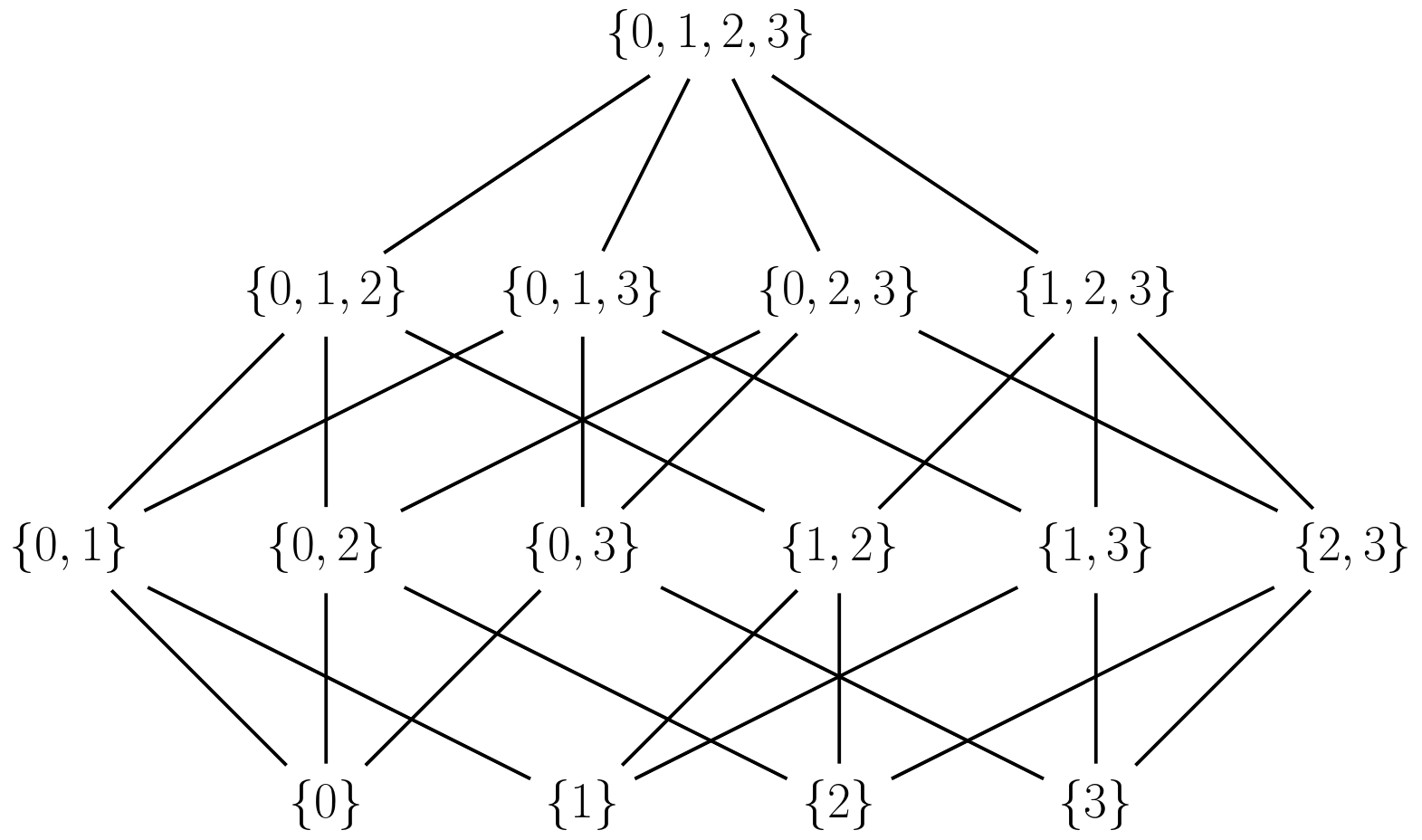}
        \caption{$\cP^{*}_3$}
    \end{subfigure}
    \hspace{2ex}
    \begin{subfigure}{0.45\linewidth}
        \centering
        \includegraphics[width=0.9\linewidth]{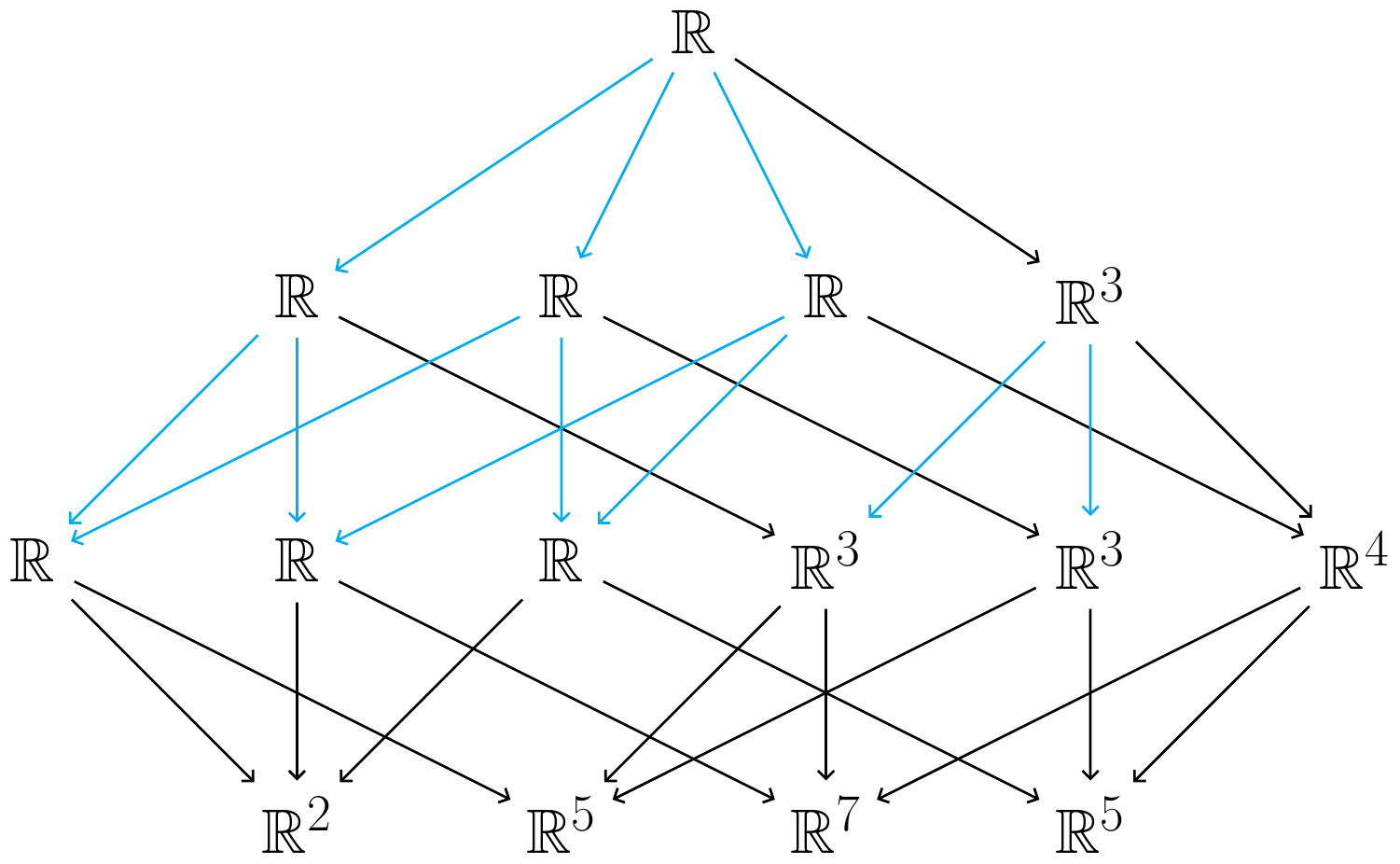}
        \caption{$\cM$}
    \end{subfigure}
    \caption{(A) Hasse diagram of the underlying poset $\cP^{*}_3$. (B) A persistence module over $\cP^{*,\op}_{3}$, where blue arrows indicate isomorphisms.}
    \label{Fig:DecompositionNExample}
\end{figure}

\begin{figure}[htb!]
    \centering
    \includegraphics[width=0.9\linewidth]{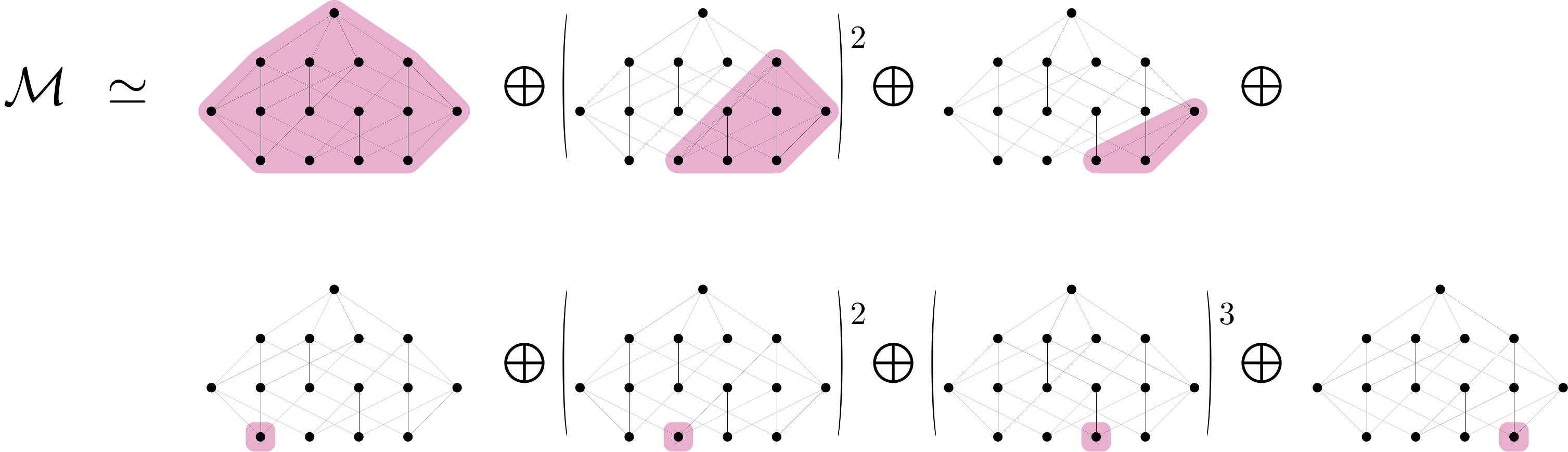}
    \caption{Interval decomposition of the persistence module in Figure~\ref{Fig:DecompositionNExample}.}
    \label{Fig:DecompositionNBarcode}
\end{figure}

\section{Robustness for cellular sheaves: thickness and cohesion}\label{Sec:GrosorHaces}

\subsection{Thickness, sheaf cohomology and persistence}\label{ss:CohHaz}\quad 

Building on sheaf persistence of topological type, in this section we adapt the analysis of thickness to the framework of cellular sheaves. This approach allows us to quantify the robustness of local data as a function of the thickness of the underlying complex. To this end, we evaluate the response of sheaf cohomology when restricted to the corresponding coskeleta (see Section~\ref{s:pre1}).
\begin{definition}
Let $\cF\colon P_X\to \vect$ be a cellular sheaf on a simplicial complex $X$, and let $n,q\in \bbN$. We define the $n$-th cohomology space of $q$-thickness of $\cF$, denoted $H^{n,q}(X;\cF)$, as the $n$-th cohomology space of the restriction of the cellular sheaf $\cF$ to the $q$-coskeleton of $X$; that is,
    \[
    H^{n,q}(X;\cF)\coloneqq H^n(X^q;\cF_{|_{X^q}})\, .
    \]
\end{definition}
As with thick Betti numbers, taking $q=0$ in the preceding definition recovers the cohomology spaces of the sheaf: $H^{n,0}(X;\cF)=H^n(X;\cF)$. Furthermore, if $\cF$ is the constant sheaf $\Bbbk$, we recover the thick Betti numbers of \cite{Pablo_Dani_Dario_25} (see Section~\ref{s:pre1}):
\[
\beta^{n,q}(X;\Bbbk)=\dim H^{n,q}(X;\Bbbk)\,.
\]

A complete analysis of the cohomology spaces of thickness requires a dynamic perspective. Starting from the cofiltration by coskeleta
\[
X=X^0\supseteq X^1\supseteq X^2\supseteq \cdots \supseteq X^{\dim X}\,,
\]
and restricting the sheaf $\cF$ on $X$ to the successive subcomplexes, we obtain the following persistence module of topological type:
\[
H^{n,0}(X;\cF)\to H^{n,1}(X;\cF)\to H^{n,2}(X;\cF)\to \dots\to H^{n,\dim X}(X;\cF)\to 0\, .
\]
In this way, the barcode of the module $H^{n,\bullet}(X;\cF)$ offers a description of the interplay between the underlying topology of the simplicial complex and the algebraic structure of the sheaf.

\begin{example}
Consider the cellular sheaf $\cF$ on the simplicial complex $X$ represented in Figure~\ref{Fig:Example5}. The sheaf assigns the stalk $\R^2$ to each simplex, with the restriction morphisms being the identity on almost all incidences. The only exceptions are the linear maps associated with the inclusions of $v_1$ and $v_2$ into the edge $(v_1,v_2)$, which are defined respectively by the matrices: $\begin{pmatrix}
    1 & 0\\
    0 & 0
\end{pmatrix}$ and $\begin{pmatrix}
    1& 0 \\ 
    1 & 1 
\end{pmatrix}$. 
\begin{figure}[htb!]
    \centering
    \begin{subfigure}[t]{0.45\textwidth}
        \centering
        \includegraphics[width=0.9\linewidth]{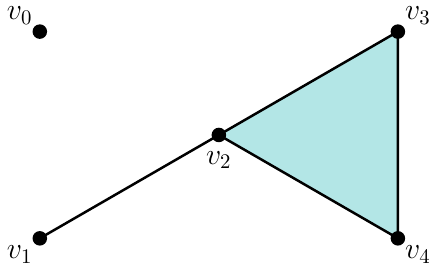}
        \caption{ }
    \end{subfigure}
    \hfill
    \begin{subfigure}[t]{0.45\textwidth}
        \centering
        \includegraphics[width=0.92\linewidth]{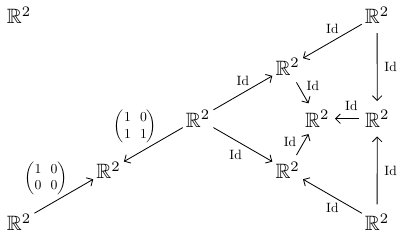}
        \caption{ }
    \end{subfigure}    
    \caption{Simplicial complex $X$ (A) and cellular sheaf $\cF$ (B). For clarity, the restriction morphisms associated with the inclusions of vertices into the triangle have been omitted, since these are defined by composition.}
    \label{Fig:Example5}
\end{figure}

Let us study the evolution of cohomology classes in degree $0$ along the cofiltration by coskeleta. To this end, let $\{x_i,y_i\}$ be the natural basis of the stalk of $\cF$ at $v_i$, $\cF(v_i)=\R^2$.

First, the space $H^0(X;\cF)$ consists of the vectors of the form $\big((\alpha_i,\beta_i)\big)_{0\leq i\leq 4}\in \prod_{i=0}^4 \cF(v_i)$ satisfying the following compatibility conditions on the edges:
\[
\begin{gathered}
    \begin{cases}
    \cF_{v_1\face e_{12}}(\alpha_1,\beta_1)=\cF_{v_2\face e_{12}}(\alpha_2,\beta_2)\\[1.5ex]
    \cF_{v_2\face e_{23}}(\alpha_2,\beta_2)=\cF_{v_3\face e_{23}}(\alpha_3,\beta_3)\\[1.5ex]
    \cF_{v_2\face e_{24}}(\alpha_2,\beta_2)=\cF_{v_4\face e_{23}}(\alpha_4,\beta_4)\\[1.5ex]
    \cF_{v_3\face e_{34}}(\alpha_3,\beta_3)=\cF_{v_4\face e_{34}}(\alpha_4,\beta_4)
\end{cases}\iff 
\begin{cases}
   (\alpha_1,0)=(\alpha_2,\alpha_2+\beta_2)\\[1.5ex]
    (\alpha_2,\beta_2)=(\alpha_3,\beta_3)\\[1.5ex]
   (\alpha_2,\beta_2)=(\alpha_4,\beta_4)\\[1.5ex]
    (\alpha_3,\beta_3)=(\alpha_4,\beta_4)
\end{cases}
\end{gathered}
\]
Therefore: $H^0(X;\cF)=\{(\alpha_i,\beta_i)\in \prod_{i=0}^4 \cF(v_i) : \alpha_1=\alpha_2=\alpha_3=\alpha_4=-\beta_2=-\beta_3=-\beta_4\}$. Repeating these calculations for the restrictions of the sheaf $\cF$ to the $1$-coskeleton and to the $2$-coskeleton, we obtain:
\begin{align*}
    H^{0,0}(X;\cF)& = \R^4=\langle x_0,\,y_0,\,y_1,\,x_1+x_2-y_2+x_3-y_3+x_4-y_4\rangle\, ,\\[2ex]
    H^{0,1}(X;\cF)&=\R^2= \langle y_1,\, x_1+x_2-y_2+x_3-y_3+x_4-y_4\rangle\, ,\\[2ex]
    H^{0,2}(X;\cF)& =\R^2=\langle x_2+x_3+x_4,\, y_2+y_3+y_4\rangle\, .
\end{align*}

The evolution of the basis vectors of $H^{0,0}(X;\cF)$ is compactly summarized in the barcode of $H^{0,\bullet}(X;\cF)$ (Figure~\ref{Fig:Example5CodigoBarras}). The top two bars correspond to the evolution of the basis vectors of the stalk of $\cF$ at the isolated vertex, $v_0$, which is removed when taking the $1$-coskeleton of $X$. The longest interval traces the persistence of the vector $x_1+x_2-y_2+x_3-y_3+x_4-y_4$. Finally, the removal of the edge $(v_1,v_2)$ in the passage to the $2$-coskeleton annihilates $y_1$, while also lifting the restriction $\alpha_i = -\beta_i$ for each $i=2,3,4$, causing the birth of a new cohomology class at stage~$2$.
\end{example}
\begin{figure}[htb!]
        \centering
        \includegraphics[width=0.65\linewidth]{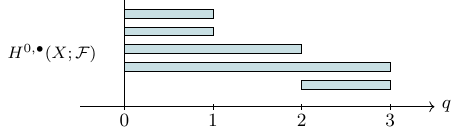}    
    \caption{Barcode of the persistence module $H^{0,\bullet}(X;\cF)$.}
    \label{Fig:Example5CodigoBarras}
\end{figure}

As the previous example illustrates, the barcode associated with thickness cohomology provides two complementary perspectives: first, it identifies the dimension of the simplices supporting nontrivial cohomology classes; and second, it indicates the dimension of the facets that obstruct the existence of additional cohomology classes.

\subsection{Cohesion and sheaf cohomology}\label{Subsec:CohesionHaces}\quad 

In this section we formalize the concept of cohesion for cellular sheaves, a construction that allows us to analyze how local data are algebraically related through higher-order adjacencies. Since removing simplices of specific dimensions generally breaks the simplicial structure, as in Section~\ref{s:pre1}, we will approach the study by means of sheaf theory over finite topological spaces.

Let $X$ be a simplicial complex of dimension $N$, let $\cF$ be a cellular sheaf on $X$, and let $\h\in \cP^*_N$.
\begin{definition}
We define the $n$-th cohomology space of $\h$-cohesion of $\cF$, denoted by $H^{n,\h}(X;\cF)$, as the $n$-th cohomology space of the restriction of the sheaf $\widehat{\cF}$ over $P_{X}$ to the $\h$-face poset of $X$; that is,
    \[
    H^{n,\h}(X;\cF)\coloneqq H^n(P_X^\h;\widehat{\cF}_{|_{P_X^{\h}}})\, .
    \]
\end{definition}
Notice that by taking $\h=[N]$ in the preceding definition we obtain the cohomology of the cellular sheaf (Theorem~\ref{Thm:EquivalenciaCohomologiacelularyhaces}):
\[
H^{n,[N]}(X;\cF)\simeq H^n(X;\cF)\, .
\]
Moreover, considering $\cF=\Bbbk$, we recover the cohesion Betti numbers of \cite{Pablo_Dani_Dario_25} (Section~\ref{s:pre1}):
\begin{proposition}
    The sheaf cohomology of the constant sheaf $\Bbbk$ over the $\h$-face poset is isomorphic to the simplicial cohomology with coefficients in $\Bbbk$ of the order complex of $P_X^{\h}$:
    \[
    H^n(P_X^\h;\Bbbk) \simeq H^n(\mathcal{K}(P_X^\h);\Bbbk)\,.
    \]
    As a consequence, $\dim H^{n,\h}(X;\Bbbk) = \beta^{n,\h}(X;\Bbbk)$.
\end{proposition}
\begin{proof}
It follows from the fact that the cochain complex induced by the standard resolution of the constant sheaf $\Bbbk$ on $P_X^{\h}$ (Equation~\eqref{Eq:CoboundaryStandardResolution}) coincides with the simplicial cochain complex of $\mathcal{K}(P_X^\h)$ with coefficients in $\Bbbk$, provided the simplices are oriented according to the natural ordering of the chains.
\end{proof}

Recall that cohomology in degree $0$ corresponds to the global sections of the cellular sheaf, $H^0(X;\cF)=\Gamma(X;\cF)$, that is, to globally compatible distributions of information over the simplicial complex. In this new context, cohesion cohomology makes it possible to capture new forms of harmony in a network beyond pairwise compatibility. To formalize this idea, we begin by generalizing the notion of global section by introducing $(h_0,h_1)$-global sections.

\begin{definition}
    Let $0\leq h_0<h_1$. An $(h_0,h_1)$-global section of $\cF$ is a collection of vectors $\{x_\sigma\}_{\sigma\in S^{h_0}(X)}$, where $x_\sigma\in \cF(\sigma)$ for each $\sigma\in S^{h_0}(X)$, such that for each simplex $\tau$ with $\dim\tau\geq h_1$ and each pair of $h_0$-simplices $\sigma,\sigma'\face \tau$ we have $\cF_{\sigma\face\tau}(x_{\sigma})=\cF_{\sigma'\face\tau}(x_{\sigma'})$. The set of $(h_0,h_1)$-global sections of $\cF$ forms a vector space, which we denote by $\Gamma^{h_0,h_1}(X;\cF)$.
\end{definition}

In particular, the space of $(0,1)$-global sections coincides with the space of global sections of the sheaf, $\Gamma^{0,1}(X;\cF)=\Gamma(X;\cF)=H^0(X;\cF)$.

The $(h_0,h_1)$-global sections extend the idea of agreement in a network captured by global sections to a higher-order compatibility. In this sense, each $(h_0,h_1)$-global section corresponds to a distribution of information over the $h_0$-simplices that is shared congruently in simplices of dimension greater than or equal to $h_1$, without requiring agreement in smaller communities.

\begin{example}\label{Ex:Consenso0hsecciones}
Consider the cellular sheaf $\cF$ in Figure~\ref{Fig:02GlobalSectionA}, whose restriction morphisms are the identity map for the inclusions of vertices into edges, and the linear map
    \[
    \begin{array}{ccc}
        \R^2 & \longrightarrow & \R \\
         (x,y) & \longmapsto & x+y
    \end{array}\, 
    \]
for the inclusions of vertices and edges into the triangle. Denote by $\{v_0,v_1,v_2\}$ the vertices of the triangle and by $\{x_i,y_i\}$ the natural basis of $\cF(v_i)=\R^2$. The space of global sections of the sheaf is $\Gamma(X;\cF)=\R^2$, with basis the pair of vectors \{$x_0+x_1+x_2, y_0+y_1+y_2$\}.
 
On the other hand, the space of $(0,2)$-global sections of the cellular sheaf is
\begin{align*}
    \Gamma^{0,2}(X;\cF)&=\{(\alpha_i,\beta_i)\in \prod_{i=0}^2 \R^2 : \alpha_0+\beta_0=\alpha_1+\beta_1=\alpha_2+\beta_2\}\\[2ex]
    &= \langle x_0+y_1+y_2,\,y_0+y_1+y_2,\,x_1-y_1,\,x_2-y_2\rangle=  \R^4\, .
\end{align*}
Figure~\ref{Fig:02GlobalSectionB} shows an example of a $(0,2)$-global section of the sheaf that is not a global section: the images of the vectors associated with the vertices coincide over the $2$-simplex, but they do not coincide over each edge.
\end{example}

\begin{figure}[htb!]
    \centering
    \begin{subfigure}[b]{0.47\textwidth}
        \centering
        \includegraphics[width=0.83\textwidth]{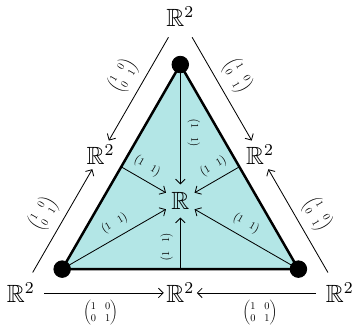}
        \caption{} 
        \label{Fig:02GlobalSectionA}
    \end{subfigure}
    \hfill 
    \begin{subfigure}[b]{0.47\textwidth}
        \centering
        \includegraphics[width=0.87\textwidth]{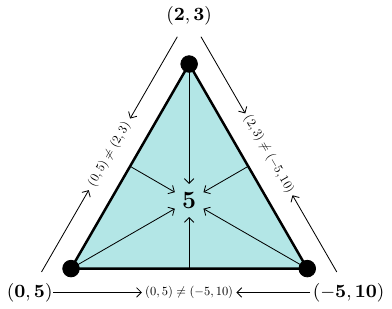}
        \vspace{2.5ex}
        \caption{} 
        \label{Fig:02GlobalSectionB}
    \end{subfigure}
    \caption{Cellular sheaf on a $2$-simplex (A) and $(0,2)$-global section (B).}\label{Fig:02GlobalSection}
\end{figure}

\begin{proposition}\label{Prop:h0h1sections}
A vector $x\in C^{h_0}(X;\cF)$ is an $(h_0,h_1)$-global section if and only if for each $h_1$-simplex $\tau\in X$ and each pair of $h_0$-simplices $\sigma,\sigma'\face \tau$ we have $\cF_{\sigma\face\tau}(x_{\sigma})=\cF_{\sigma'\face\tau}(x_{\sigma'})$.
\end{proposition}
\begin{proof}
The direct implication is immediate from the definition of $\Gamma^{h_0,h_1}(X;\cF)$. For the converse, assume that a vector $x\in C^{h_0}(X;\cF)$ satisfies the second condition in the statement, and let us show that $x$ is an $(h_0,h_1)$-global section. Let $\mu\in X$ be a simplex of dimension greater than or equal to $h_1$, and let $\sigma$ and $\sigma'$ be two $h_0$-faces of $\mu$. There exists a sequence $\{\sigma_0,\sigma_1,\dots,\sigma_n\}$ of $h_0$-faces of $\mu$ and a sequence $\{\tau_1,\dots,\tau_n\}$ of $h_1$-faces of $\mu$ such that
      \[
    \sigma=\sigma_0\face\tau_1 \trianglerighteqslant\sigma_1\face \tau_2\trianglerighteqslant\sigma_2\face \dots \face\tau_n\trianglerighteqslant\sigma_n=\sigma'\, .
    \]
    By hypothesis, $\cF_{\sigma_{i-1}\face\tau_i}(x_{\sigma_{i-1}})=\cF_{\sigma_{i}\face\tau_i}(x_{\sigma_i})$ for all $1\leq i\leq n$. Applying the morphism $\cF_{\tau_i\face\mu}$ to both sides of the equality, and by functoriality of the sheaf, we obtain:
    \[
    \cF_{\sigma_{i-1}\face\mu}(x_{\sigma_{i-1}})=\cF_{\sigma_{i}\face\mu}(x_{\sigma_{i}}) \,,\quad \text{for all  }1\leq i\leq n\, .
    \]
    Since $\sigma_0=\sigma$ and $\sigma_n=\sigma'$, we conclude that $\cF_{\sigma\face\mu}(x_{\sigma})=\cF_{\sigma'\face\mu}(x_{\sigma'})$, that is, $x\in \Gamma^{h_0,h_1}(X;\cF)$.
\end{proof}

\begin{proposition}\label{Prop:InclusionGlobalSections}
    For any $0\leq h_0<h_1< h_1'$, we have $\Gamma^{h_0,h_1}(X;\cF)\subseteq \Gamma^{h_0,h_1'}(X;\cF)$. Moreover, if each $h_1$-simplex $\tau\in X$ is a face of some $h_1'$-simplex $\mu\in X$ such that the restriction morphism $\cF_{\tau\face \mu}$ is injective, then
    \[
    \Gamma^{h_0,h_1}(X;\cF)=\Gamma^{h_0,h_1'}(X;\cF)\,.
    \]
\end{proposition}
\begin{proof}
         The inclusion $\Gamma^{h_0,h_1}(X;\cF)\subseteq \Gamma^{h_0,h_1'}(X;\cF)$ follows directly from the definition.
    
    As for the equality, let $x\in\Gamma^{h_0,h_1'}(X;\cF)$ and let us show that $x\in \Gamma^{h_0,h_1}(X;\cF)$. By Proposition~\ref{Prop:h0h1sections}, it suffices to prove that $\cF_{\sigma\face\tau}(x_\sigma)=\cF_{\sigma'\face\tau}(x_{\sigma'})$ for any $h_1$-simplex $\tau\in X$ and each pair of $h_0$-faces $\sigma,\sigma'\face \tau$.

    Suppose, for contradiction, that $\cF_{\sigma\face\tau}(x_\sigma)\neq\cF_{\sigma'\face\tau}(x_{\sigma'})$. By hypothesis, there exists an $h_1'$-simplex $\mu$ containing $\tau$ such that $\cF_{\tau\face \mu}$ is injective. Thus,
    \[
    \cF_{\tau\face \mu}\cF_{\sigma\face\tau}(x_\sigma)\neq\cF_{\tau\face \mu}\cF_{\sigma'\face\tau}(x_{\sigma'})\implies \cF_{\sigma\face\mu}(x_\sigma)\neq\cF_{\sigma'\face\mu}(x_{\sigma'})\, ,
    \]
    which contradicts $x\in \Gamma^{h_0,h_1'}(X;\cF)$.
\end{proof}

\begin{proposition}\label{Prop:CohomGlobalSections}
     For each $\h=\{h_0,h_1,\dots, h_m\}$, the $0$-th cohomology space of $\h$-cohesion of $\cF$ is isomorphic to the space of $(h_0,h_1)$-global sections of $\cF$:
    \[
    H^{0,\h}(X;\cF)\simeq \Gamma^{h_0,h_1}(X;\cF)\, .
    \]
\end{proposition}
\begin{proof}
First, with the notation in Section~\ref{Sec:HacesPosets}, note that $\widehat{\cF}_{|_{P_X^\h}}=\widehat{\cF_{|_{P_X^\h}}}$. Therefore,
\[
    H^{0,\h}(X;\cF)=H^0(P_X^\h;\widehat{\cF}_{|_{P_X^\h}})=\widehat{\cF_{|_{P_X^\h}}}(P_X^\h)=\plim{\sigma\in P_X^\h}\cF(\sigma)\, .
\]
Using the characterization of the limit of vector spaces, we obtain:
\[
\plim{\sigma\in P_X^\h}\cF(\sigma)=\{x\in \displaystyle\prod_{\dim\sigma\in \h}\cF(\sigma) : \cF_{\sigma\face\tau}(x_\sigma)=x_\tau \text{ for all } \sigma\face \tau \text{ with } \dim\sigma,\dim\tau\in \h\}\, .
\]
By the compatibility condition in the preceding equation, the value of $x_\tau$ is determined by the value of $x$ on any of its $h_0$-faces. Thus, the preceding description of the limit is equivalent to the following:
\begin{align*}
    \plim{\sigma\in P_X^\h}\cF(\sigma)\simeq \ & \{x\in \displaystyle\prod_{\dim\sigma=h_0}\cF(\sigma) :  \cF_{\sigma\face\tau}(x_\sigma)=\cF_{\sigma'\face \tau}(x_{\sigma'})\text{ for all } \sigma,\sigma'\face\tau \\[1.5ex]  & \text{ with }\dim \sigma=\dim\sigma'=h_0 \text{ and } \dim \tau= h_i\,,\text{ for some } 1\leq i\leq m\}\, .
\end{align*}
By the same argument as in Proposition~\ref{Prop:h0h1sections}, it suffices to check these equalities on the \mbox{$h_1$-simplices}:
\[
\plim{\sigma\in P_X^\h}\cF(\sigma)\simeq \{x\in \displaystyle\prod_{\dim\sigma=h_0}\cF(\sigma) :  \cF_{\sigma\face\tau}(x_\sigma)=\cF_{\sigma'\face \tau}(x_{\sigma'}) \text{ for all } \sigma,\sigma'\face\tau \text{ with } \dim \tau=h_1\}\, .
\]
Finally, this space coincides with the characterization of the space of $(h_0,h_1)$-global sections given in Proposition~\ref{Prop:h0h1sections}.
\end{proof}

\section{Structure theorems for cohesion persistence}\label{ss:cohesivepersistence}

The aim of this section is to apply the structure theorems of Section~\ref{Sec:StructureTheorems} to cohesion persistence modules of cellular sheaves. These modules are indexed by $\cP^{*,\op}_N$, making them multiparameter persistence modules. The abstract results of Section~\ref{Sec:StructureTheorems} provide interval decompositions for certain modules over this poset, but their application to sheaf cohesion requires a separate verification of the hypotheses in geometric terms.

We first introduce the construction of cohesion persistence for a cellular sheaf. We then prove that, in dimension $2$, the corresponding modules are interval-decomposable degreewise: $H^{0,\bullet}(X;\cF)$ under a local injectivity assumption, while $H^{1,\bullet}(X;\cF)$ and $H^{2,\bullet}(X;\cF)$ decompose for any cellular sheaf. Finally, we extend the decomposition in degree $0$ to higher-dimensional simplicial complexes whose connected components of dimension greater than $2$ are pure and whose sheaf restriction morphisms are injective.

Let $\h=\{h_0,\dots,h_m\}\in \cP^{*}_N$. The family of $\h$-face posets defines a filtration
\[
    P_X^\bullet\colon \cP^{*}_N\longrightarrow \pos,
    \qquad
    \h\longmapsto P_X^\h,
\]
whose maximal element is the face poset $P_X^{[N]}=P_X$. Given a cellular sheaf $\cF$ on $X$, we restrict the associated sheaf $\widehat{\cF}$ to each $P_X^\h$. For every $\h\leq \h'$ in $\cP^{*}_N$, the inclusion $P_X^\h\subseteq P_X^{\h'}$ induces, by functoriality of sheaf cohomology, a linear map
\[
H^{n}_{\h,\h'}\colon
H^n(P_X^{\h'};\widehat{\cF}_{|_{P_X^{\h'}}})
\longrightarrow
H^n(P_X^{\h};\widehat{\cF}_{|_{P_X^{\h}}}) .
\]
These maps are compatible with composition, and hence define a persistence module of topological type
\[
    H^{n,\bullet}(X;\cF)\colon
    \cP^{*,\op}_N\longrightarrow \vect,
    \qquad
    \h\longmapsto
    H^n(P_X^\h;\widehat{\cF}_{|_{P_X^\h}}).
\]

The case $N=1$ reduces to zigzag persistence, so the interval decomposition follows from Gabriel's theorem (Theorem~\ref{Thm:Gabriel}). For $N=2$, we prove a structure theorem by cohomological degree: $H^{0,\bullet}(X;\cF)$ decomposes under a local injectivity hypothesis, $H^{1,\bullet}(X;\cF)$ decomposes for every cellular sheaf, and $H^{2,\bullet}(X;\cF)$ is supported at a single point. 

We begin with several technical results that will be used in the proofs of the structure theorems.

\begin{lemma}\label{Lem:InyectivosIsomorfismos}
    Let $\h=\{h_0,h_1,\dots,h_m\}\in \cP^{*}_N$ with $m\geq 2$, and let $\h_i=\{h_0,\dots,\widehat{h}_i,\dots,h_m\}$. Consider the linear map induced by the inclusion $\h_i\leq \h$:
    \begin{equation}\label{Eq:MorfismoHnhhi}
        H^0_{\h_i,\h}\colon H^{0,\{h_0,\dots,h_m\}}(X;\cF)\longrightarrow H^{0,\{h_0,\dots,\widehat{h}_i,\dots,h_m\}}(X;\cF)\, .
    \end{equation}
    Then the following properties hold:
    \begin{itemize}
        \item The map $H_{\h_i,\h}^0$ is injective for $i=1$, and is an isomorphism for all $2\leq i\leq m$.
        \item If for every $h_1$-simplex $\sigma$ there exists an $h_2$-simplex $\tau$ such that $\sigma\face\tau$ and the restriction map $\cF_{\sigma\face\tau}$ is injective, then $H^0_{\h_1,\h}$ is an isomorphism.
        \item If for every $\sigma\in S^{h_0}(X)$ there exists $\tau\in S^{h_1}(X)$ such that $\sigma\face \tau$ and $\cF_{\sigma\face \tau}$ is injective, then $H^0_{\h_0,\h}$ is injective.
    \end{itemize}
\end{lemma}
\begin{proof} For the first two parts, see Propositions~\ref{Prop:InclusionGlobalSections} and \ref{Prop:CohomGlobalSections}. As for the last part, recall that for each $\h\in \cP^{*}_N$,
 \[
   H^{0,\h}(X;\cF)=\{x\in \prod_{\dim\sigma\in \h}\cF(\sigma) : \cF_{\sigma\face\tau}(x_\sigma)=x_\tau \text{ for all } \sigma\face \tau \text{ with } \dim\sigma,\dim\tau\in \h\}\, .
   \]
 Moreover, for each $\h\leq \h'$, the map $H^{0}_{\h,\h'}$ is induced by the natural projection
   \[
  \prod_{\dim\sigma\in\h'}\cF(\sigma)\longrightarrow \prod_{\dim\sigma\in \h}\cF(\sigma)\, .
  \]
Now consider a vector $x\in H^{0,\h}(X;\cF)$ such that $H^0_{\h_0,\h}(x)=0$. Thus, $x_\tau=0$ for every $\tau$ such that $\dim\tau\neq h_0$. Let $\sigma\in S^{h_0}(X)$; we will show that $x_\sigma$ must also vanish.

By hypothesis, there exists an $h_1$-simplex $\tau$ such that $\sigma\face\tau$ and the map $\cF_{\sigma\face\tau}$ is injective. Since $x\in H^{0,\h}(X;\cF)$, we have $\cF_{\sigma\face\tau}(x_\sigma)=x_\tau=0$, and by the injectivity of the restriction map we conclude that $x_\sigma=0$.
\end{proof}

\begin{lemma}\label{Lem:DescomposicionComponentesConexas}
    Let $X_1,\dots,X_k$ be the connected components of $X$, and let $N_j\coloneqq \dim X_{j}$. Denote by $i_j$ the natural inclusion $\cP_{N_j}^{*,\op}\subseteq \cP^{*,\op}_N$. Then
    \begin{equation*}
        H^{n,\bullet}(X;\cF)\simeq {i_{1}}_*\,H^{n,\bullet}(X_1;\cF_{|_{X_1}})\,\oplus\,\dots\, \oplus\, {i_k}_*\,H^{n,\bullet}(X_k;\cF_{|_{X_k}})\, .
    \end{equation*}
\end{lemma}
\begin{proof}
    For each $\h\in \cP^{*}_N$, we have $P_X^{\h}=\bigsqcup P_{X_j}^\h$, and therefore
    \[
    H^{n,\h}(X;\cF)=H^{n}(P_X^\h;\widehat{\cF}_{|_{P_X^\h}})=\bigoplus_{j=1}^k H^n(P_{X_j}^\h;\widehat{\cF}_{|_{P_{X_j}^{\h}}})\, ,
    \]
    from which we deduce that
    \begin{equation}\label{Eq:DescomposicionComConexas}
        H^{n,\bullet}(X;\cF)\simeq \bigoplus_{j=1}^k H^{n}(P_{X_j}^\bullet;\widehat{\cF}_{|_{P_{X_j}^\bullet}})\,.
    \end{equation}
    On the other hand, for each $1\leq j\leq k$,
    \[
    \left({i_j}_* H^{n,\bullet}(X_j;\cF_{|_{X_j}})\right)_\h=
        \plim{i_j(\h')\geq_{\op} \h}H^{n,\h'}(X_j;\cF_{|_{X_j}})\, .
    \]
    The subset $\h_j\coloneqq \{h\in \h : h\leq N_j\}$ is the least element (with respect to the opposite order) such that $i_j(\h_j)\geq_{\op}\h$. Therefore, the natural morphism
    \[
    \plim{i_j(\h')\geq_{\op} \h}H^{n,\h'}(X_j;\cF_{|_{X_j}})\longrightarrow H^{n,\h_j}(X_j;\cF_{|_{X_j}})
    \]
    is an isomorphism. Moreover, for each $\h\leq_{\op} \overline{\h}$, the following diagram commutes:
    \[
    \begin{tikzcd}
        \plim{i_j(\h')\geq_{\op} \h}H^{n,\h'}(X_j;\cF_{|_{X_j}})\arrow[r,"\sim"]\arrow[d] & H^{n,\h_j}(X_j;\cF_{|_{X_j}})\arrow[d] \\
        \plim{i_j(\h')\geq_{\op} \overline{\h}}H^{n,\h'}(X_j;\cF_{|_{X_j}})\arrow[r,"\sim"] &  H^{n,\overline{\h}_j}(X_j;\cF_{|_{X_j}})
    \end{tikzcd}
    \]
    Taking into account that $P_{X_j}^{\h_j}=P_{X_j}^\h$, we deduce:
    \[
    H^{n,\h_j}(X_j;\cF_{|_{X_j}})=H^n(P_{X_j}^{\h_j};(\widehat{\cF_{|_{X_j}}})_{|_{P_{X_j}^{\h_j}}})=H^{n}(P_{X_j}^{\h_j};\widehat{\cF}_{|_{P_{X_j}^{\h_j}}})=H^n(P_{X_j}^\h;\widehat{\cF}_{|_{P_{X_j}^\h}})\,.
    \]
    That is, for each $1\leq j\leq k$, there exists an isomorphism of functors
    \[
    {i_j}_* H^{n,\bullet}(X_j;\cF_{|_{X_j}})\stackrel{\sim}\longrightarrow H^{n}(P_{X_j}^\bullet;\widehat{\cF}_{|_{P_{X_j}^\bullet}})\,, 
    \]
    Substituting into the isomorphism of Equation~\eqref{Eq:DescomposicionComConexas}, the conclusion follows.
\end{proof}
\begin{corollary}\label{Cor:DescomposicionComponentesConexas}
    With the notation of the preceding lemma, if $H^{n,\bullet}(X_j;\cF_{|_{X_j}})$ decomposes as a direct sum of interval modules for every $1\leq j\leq k$, then $H^{n,\bullet}(X;\cF)$ does as well.
\end{corollary}
\begin{proof}
    Suppose that for each $1\leq j\leq k$ there is a decomposition
    \begin{equation}\label{Eq:CohesiveDecompositionConnectedComponent}
        H^{n,\bullet}(X_j;\cF_{|_{x_j}})\simeq \bigoplus_{\lambda\in \Lambda_j}\Bbbk[I_{\lambda}]\, ,
    \end{equation}
    for a certain set of intervals $\{I_\lambda\}$ of $\cP_{N_j}^{*,\op}$. For each $I_\lambda$, consider the following subset of $\cP^{*,\op}_N$:
    \begin{equation}\label{Eq:IntervaloAsociado}
        \widetilde{I_\lambda}\coloneqq\{\h\in \cP^{*,\op}_N\colon \h \cap [N_j] \in I_\lambda\}\, .
    \end{equation}
    This subset is an interval:
    \begin{enumerate}
        \item $\widetilde{I}_\lambda$ is convex: let $\h,\h'\in \widetilde{I}_\lambda$, and let $\h^{\prime\prime}\in \cP^{*,\op}_N$ be such that $\h\leq_{\op}\h^{\prime\prime}\leq_{\op} \h'$. Then,
        \[
        \h\cap [N_j]\leq_{\op}\h^{\prime\prime}\cap [N_j]\leq_{\op} \h'\cap [N_j] \,.
        \]
        Since $\h\cap [N_j],\,\h^{\prime}\cap [N_j]\in I_\lambda$ and this interval is convex, it follows that $\h^{\prime\prime}\cap [N_j]\in I_\lambda$, and hence $\h^{\prime\prime}\in \widetilde{I}_\lambda$.
        \item $\widetilde{I}_\lambda$ is connected: let $\h,\,\h'\in \widetilde{I}_\lambda$. Since $\h\cap [N_j],\h'\cap [N_j]$ are elements of $ I_\lambda$, which is connected, there exists a sequence $\{\h^i\}_{i=1}^n\subseteq I_\lambda$ such that $\h^1=\h\cap [N_j]$, $\h^n=\h'\cap [N_j]$, and $\h^i\leq_{\op}\h^{i+1}$ or $\h^i\geq_{\op}\h^{i+1}$ for all $1\leq i<n$. Now, $\h^i\in \widetilde{I}_\lambda$ for all $1\leq i\leq n$, $\h\leq_{\op}\h^1$, and $\h'\leq_{\op}\h^n$, so the sequence $\{\h,\h^1,\dots,\h^n,\h'\}$ forms a path in $\widetilde{I}_\lambda$ connecting $\h$ and $\h'$.
    \end{enumerate}
    Moreover, one has ${i_j}_*\Bbbk[I_\lambda]=\Bbbk[\widetilde{I}_{\lambda}]$ because for each $\h\in \cP^{*,\op}_N$,
    \[
    ({i_j}_*\Bbbk[I_\lambda])_{\h}=\plim{i_j(\h')\geq_{\op}\h}\Bbbk[I_\lambda]_{\h'}=\Bbbk[I_\lambda]_{\h\cap [N_j]}=\Bbbk[\widetilde{I}_\lambda]_{\h}\,.
    \]
    Substituting into the decomposition of Lemma~\ref{Lem:DescomposicionComponentesConexas}, we obtain a decomposition into interval modules:
    $$H^{n,\bullet}(X;\cF)\simeq \bigoplus_{j=1}^k\bigoplus_{\lambda\in \Lambda_j}\, \Bbbk[\widetilde{I_\lambda}]\, .$$
\end{proof}

\begin{theorem}\label{Thm:DescomposicionCohesiveDim2}
Let $X$ be a $2$-dimensional simplicial complex, and let $\cF$ be a cellular sheaf on $X$. Then the following statements hold:

\begin{enumerate}[label=(\roman*)]
    \item Suppose that for every vertex $v$ in a $2$-dimensional connected component of $X$, there exists an incident edge $e$ such that the map $\cF_{v\face e}$ is injective. Then the cohesion persistence module
    \[
        H^{0,\bullet}(X;\cF)\colon
        \cP^{*,\op}_2
        \longrightarrow \vect
    \]
    decomposes into interval persistence modules.
    \item The cohesion persistence module
    \[
        H^{1,\bullet}(X;\cF)\colon
        \cP^{*,\op}_2
        \longrightarrow \vect
    \]
    decomposes into interval persistence modules.
    \item The cohesion persistence module
    \[
        H^{2,\bullet}(X;\cF)\colon
        \cP^{*,\op}_2
        \longrightarrow \vect
    \]
    is supported solely at the index $\h=\{0,1,2\}$, making its interval decomposition trivial.
    \end{enumerate}
\end{theorem}
\begin{remark}
The injectivity hypothesis in Theorem~\ref{Thm:DescomposicionCohesiveDim2}(i) should be understood as a local faithfulness condition. It does not require every restriction map of the sheaf to be injective. Rather, it asks that each vertex in a $2$-dimensional connected component has at least one incident edge along which its local data are transmitted without collapse. Equivalently, if $\cF_{v\face e}$ is injective, then two private states of $v$ that induce the same expression on $e$ must already be equal. In particular, if a state is expressed as zero through this channel, then the state itself is zero. Thus, in applications where vertices represent agents and edges represent communication contexts, the condition says that every agent has at least one faithful communication channel. This is a natural non-degeneracy assumption rather than a global injectivity requirement on the whole sheaf.
\end{remark}

\begin{proof}[Proof of (i)]
By Corollary~\ref{Cor:DescomposicionComponentesConexas}, it suffices to prove that $H^{0,\bullet}(X_j;\cF_{|_{X_j}})$ decomposes into interval modules for each connected component $X_j$ of $X$. If the connected component $X_j$ has dimension $0$, the decomposition is trivial, and if it has dimension $1$, the decomposition exists by Gabriel's theorem (Theorem~\ref{Thm:Gabriel}). Thus, suppose that $X$ is connected and of dimension $2$.

The persistence module $H^{0,\bullet}(X;\cF)$ has the form:
    \begin{equation}\label{Eq:DiagramStructureTheoremCohesive1}
    \includegraphics[width=0.65\linewidth]{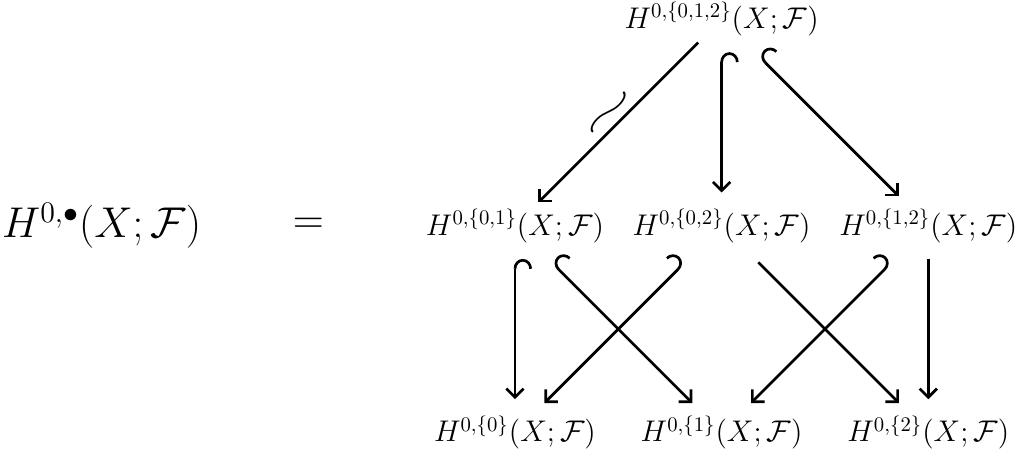}
    \end{equation}
    Indeed, following Lemma~\ref{Lem:InyectivosIsomorfismos}, we know that, in general, the map
    \[
    H^0_{\{0,1\},\{0,1,2\}}\colon H^{0,\{0,1,2\}}(X;\cF)\longrightarrow H^{0,\{0,1\}}(X;\cF)
    \]
    is an isomorphism, and the following maps are injective:
    \[
    \begin{aligned}
        H^0_{\{0,2\},\{0,1,2\}} &\colon H^{0,\{0,1,2\}}(X;\cF)\longrightarrow H^{0,\{0,2\}}(X;\cF)\,, \\[1.5ex]
        H^0_{\{0\},\{0,1\}}     &\colon H^{0,\{0,1\}}(X;\cF)\longrightarrow H^{0,\{0\}}(X;\cF)\,, \\[1.5ex]
        H^0_{\{0\},\{0,2\}}     &\colon H^{0,\{0,2\}}(X;\cF)\longrightarrow H^{0,\{0\}}(X;\cF)\,, \\[1.5ex]
        H^0_{\{1\},\{1,2\}}     &\colon H^{0,\{1,2\}}(X;\cF)\longrightarrow H^{0,\{1\}}(X;\cF)\,,
    \end{aligned}
    \]
    Moreover, since $X$ is connected of dimension $2$ and satisfies the injectivity condition from vertices to edges, by Lemma~\ref{Lem:InyectivosIsomorfismos}, the maps
    \[
    \begin{aligned}
        H^{0}_{\{1,2\},\{0,1,2\}} & \colon H^{0,\{0,1,2\}}(X;\cF)\longrightarrow H^{0,\{1,2\}}(X;\cF)\,,\\[1.5ex]
        H^{0}_{\{1\},\{0,1\}} & \colon H^{0,\{0,1\}}(X;\cF)\longrightarrow H^{0,\{1\}}(X;\cF)\,,
    \end{aligned}
    \]
    are also injective.

    Therefore, $H^{0,\bullet}(X;\cF)$ satisfies the hypotheses of Theorem~\ref{Thm:Decomposition2} for the index $i=2$. Consequently, it decomposes into a direct sum of interval modules.
\end{proof}
\begin{proof}[Proof of (ii)]
    The module $H^{1,\bullet}(X;\cF)$ has the form
    \begin{equation*}
    \includegraphics[width=0.7\linewidth]{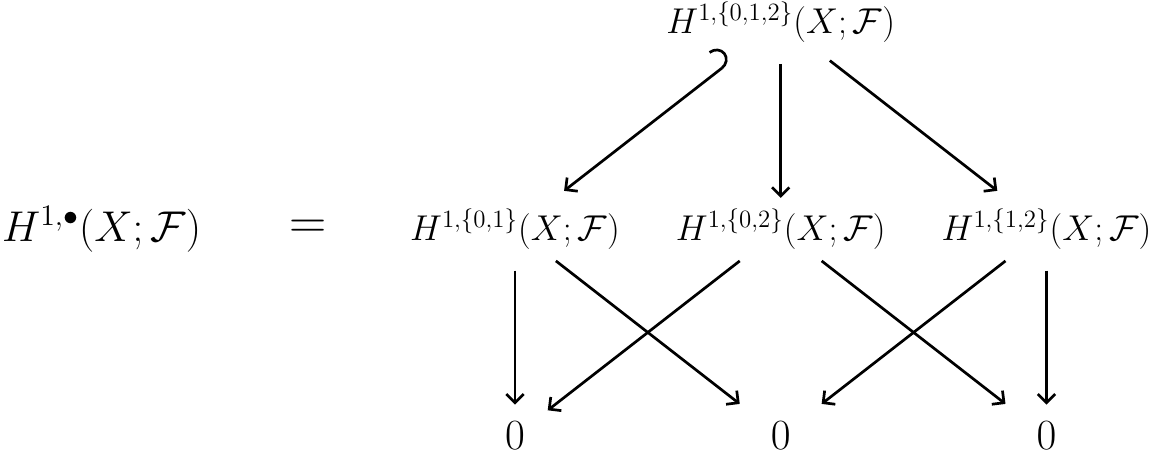}
    \end{equation*}
    Since the morphism $H^1_{\{0,1\},\{0,1,2\}}\colon H^{1,\{0,1,2\}}(X;\cF)\to H^{1,\{0,1\}}(X;\cF)$ is injective, by taking a subspace $\cM_{\{0,1\}}$ such that $H^{1,\{0,1\}}(X;\cF)=\img H^1_{\{0,1\},\{0,1,2\}}\oplus \cM_{\{0,1\}}$, we obtain the decomposition
\begin{equation*}\label{Eq:DiagramStructureTheoremDim1Desc}
        \centering
    \includegraphics[width=0.9\linewidth]{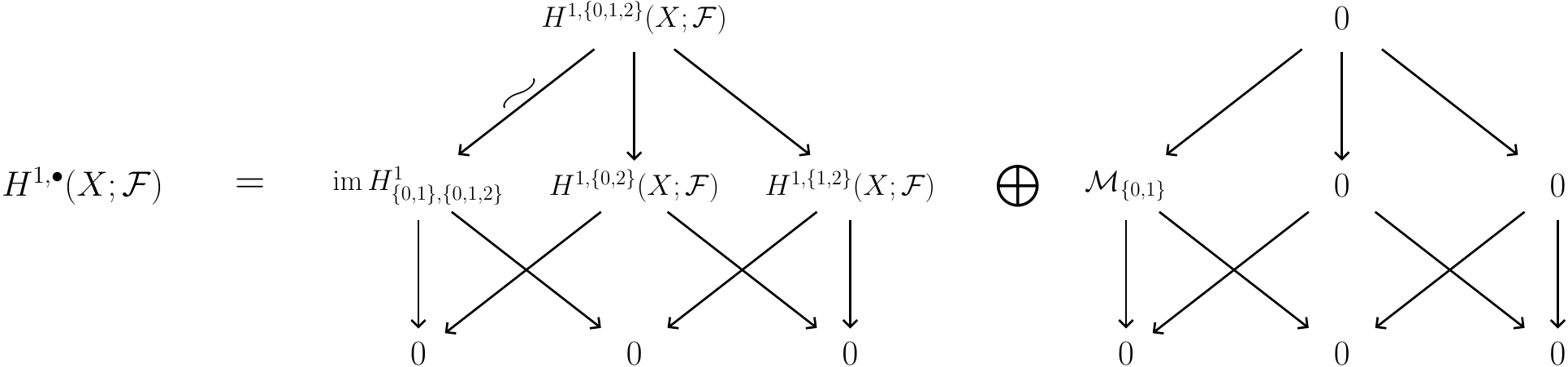}
    \end{equation*}
    To conclude, the decomposition of the first persistence module in the preceding direct sum reduces to that of the persistence module
    \[
    \begin{tikzcd}[column sep=4ex]
        & H^{1,\{0,1,2\}}(X;\cF)\arrow[rd]\arrow[ld]\\
        H^{1,\{0,2\}}(X;\cF) & & H^{1,\{1,2\}}(X;\cF)
    \end{tikzcd}
    \]
    This is a zigzag module, and the conclusion follows by Gabriel's theorem (Theorem~\ref{Thm:Gabriel}).
\end{proof}
\begin{proof}[Proof of (iii)]
If $\h\neq \{0,1,2\}$, then the poset $P_X^\h$ has chains of length at most $1$. Hence
\[
    H^{2,\h}(X;\cF)=0.
\]
Therefore, the persistence module $H^{2,\bullet}(X;\cF)$ is supported solely at the index $\{0,1,2\}$. This makes its decomposition trivial, consisting entirely of copies of the interval module $\Bbbk[\{\{0,1,2\}\}]$. Letting $n\coloneqq \dim H^{2,\{0,1,2\}}(X;\cF)$, we obtain
\[
    H^{2,\bullet}(X;\cF)
    \simeq
    \Bbbk[\{\{0,1,2\}\}]\,\oplus\,\stackrel{n}{\cdots}\,\oplus \,\Bbbk[\{\{0,1,2\}\}].
\]
\end{proof}

The decomposition of $H^{0,\bullet}(X;\cF)$ into interval modules extends to higher dimensions under additional structural hypotheses on both the simplicial complex and the cellular sheaf.

\begin{theorem}\label{Thm:DescomposicionPN}
   Let $X$ be a simplicial complex of dimension $N$ such that its connected components of dimension greater than $2$ are pure (not necessarily of dimension $N$), and let $\cF$ be a cellular sheaf on $X$ with injective restriction maps. Then the persistence module $H^{0,\bullet}(X;\cF)$ decomposes into interval persistence modules.
\end{theorem}

\begin{remark}
The purity assumption in Theorem~\ref{Thm:DescomposicionPN} is componentwise. We do not require the whole complex to be pure of a single dimension; components of different dimensions may coexist, and the condition only applies to connected components of dimension greater than $2$. This hypothesis is common in higher-order network models where the relevant interactions are generated by facets of a fixed order. In network science, several standard higher-order models are built precisely from $d$-dimensional simplices, or from ensembles of simplicial complexes with prescribed higher-order degrees~\cite{Bianconi_Rahmede16}, \cite{Bianconi_Rahmede17}, \cite{Courtney_Bianconi16}, \cite{Courtney_Bianconi17}. 
\end{remark}

\begin{proof}
    By Corollary~\ref{Cor:DescomposicionComponentesConexas}, we may assume that $X$ is connected, and by Theorem~\ref{Thm:DescomposicionCohesiveDim2}, we may assume that $X$ has dimension greater than $2$. Since $X$ is pure, Lemma~\ref{Lem:InyectivosIsomorfismos} ensures that all morphisms $H^0_{\h,\h'}$ are injective. Additionally, given that the restriction morphisms of the sheaf $\cF$ are assumed to be injective, the same lemma guarantees that the maps
    \begin{equation}\label{Eq:Isomorfismos2}
        H^0_{\h_i,\h}\colon H^{0,\{h_0,\dots,h_m\}}(X;\cF)\longrightarrow H^{0,\{h_0,\dots,\widehat{h}_i,\dots,h_m\}}(X;\cF)
    \end{equation}
    are isomorphisms for every $\h=\{h_0,\dots,h_m\}$ with $m\geq 2$ and every $1\leq i\leq m$. 
    
    Consequently, the persistence module $H^{0,\bullet}(X;\cF)$ satisfies the hypotheses of Theorem~\ref{Thm:DecompositionN}, yielding the desired decomposition.
\end{proof}

\begin{remark}
From a computational perspective, note that, by Corollary~\ref{Cor:DecompositionN_Explicit}, the interval decomposition of the cohesion persistence module on each connected component $X_j$ of dimension greater than $2$ is completely determined by the dimensions of the cohesion spaces of the form $H^{0,\{k\}}(X_j;\cF_{|_{X_j}})$ and $H^{0,\{k,k+1,\dots,\dim X_j\}}(X_j;\cF_{|_{X_j}})$.
\end{remark}

A primary example satisfying the conditions of Theorem~\ref{Thm:DescomposicionPN} is the constant cellular sheaf $\Bbbk$, which yields the following structural result for ordinary cohesion persistence modules (Equation~\eqref{Eq:CohesionModuleConstantSheaf}):
\begin{corollary}
Let $X$ be a simplicial complex of dimension $N$ such that its connected components of dimension greater than $2$ are pure. Then the cohesion persistence module
\[
H^{0}(P_X^\bullet;\Bbbk)\colon \cP_N^{*,\op} \longrightarrow \vect
\]
decomposes into interval persistence modules.
\end{corollary}

\begin{example}[An example of Theorem~\ref{Thm:DescomposicionCohesiveDim2}]\label{e:exDim2}    
Consider the cellular sheaf $\cF$ on the simplicial complex $X$ in Figure~\ref{Fig:SheafPersistentCohesion}, which consists of three connected components: a triangle, an edge, and a vertex. This sheaf satisfies the conditions of Theorem~\ref{Thm:DescomposicionCohesiveDim2} because every vertex of the triangle has at least one injective restriction morphism to an edge. Notice that, according to the conditions of the theorem, there may be edges whose pair of restriction morphisms from their vertices are not injective (in our example, this happens on the base edge of the triangle).

    \begin{figure}[htb!]
        \centering
        \includegraphics[width=0.7\linewidth]{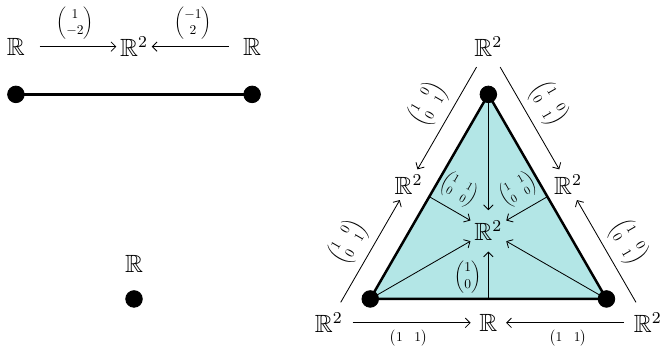}
        \caption{Cellular sheaf $\cF$ (the maps from the vertices to the triangle are obtained by composition through an edge).}
        \label{Fig:SheafPersistentCohesion}
    \end{figure}

    To compute the intervals forming the decomposition of $H^{0,\bullet}(X;\cF)$ (Figure~\ref{Fig:SheafPersistentCohesion-Barcode}), we follow the procedure described in the proof of Theorem~\ref{Thm:DescomposicionCohesiveDim2}.
    
    First, denote by $X_0$, $X_1$, and $X_2$ the connected components formed, respectively, by the $0$-simplex, the $1$-simplex, and the $2$-simplex. Then, by Lemma~\ref{Lem:DescomposicionComponentesConexas},
    \begin{equation}\label{Eq:DescomposicionBase}
         H^{0,\bullet}(X;\cF)\simeq {i_{0}}_*H^{0,\bullet}(X_0;\cF_{|_{X_0}}) \oplus {i_{1}}_*H^{0,\bullet}(X_1;\cF_{|_{X_1}}) \oplus H^{0,\bullet}(X_2;\cF_{|_{X_2}})\, .
    \end{equation}

    For the isolated vertex, $H^{0,\bullet}(X_0;\cF_{|_{X_0}})=\R[\{0\}]$. Therefore, following the proof of Corollary~\ref{Cor:DescomposicionComponentesConexas}, ${i_0}_*\R[\{0\}]=\R[\widetilde{\{0\}}]$, where $\widetilde{\{0\}}$ denotes the interval (Equation~\eqref{Eq:IntervaloAsociado}):
   \[
   \widetilde{\{0\}}=\{\h\in \cP^{*,\op}_2 :\h\cap \{0\}= \{0\}\}= \{\h\in \cP^{*,\op}_2 :\{0\}\in \h\}= \{\{0\},\{0,1\},\{0,2\},\{0,1,2\}\}
   \]
       That is,
    \begin{equation}\label{Eq:DescomposicionX0}
        {i_0}_*H^{0,\bullet}(X_0;\cF_{|_{X_0}})=\R[\{\{0\},\{0,1\},\{0,2\},\{0,1,2\}\}]\,.
    \end{equation}
    
    As for the cohesion persistence module of the edge, $H^{0,\bullet}(X_1;\cF_{|_{X_1}})$, it has the form:
    \begin{equation}\label{Eq:DiagramaX1}
            \begin{tikzcd}
        & H^{0,\{0,1\}}(X_1;\cF_{|_{X_1}}) \arrow[ld] \arrow[rd]\\
        H^{0,\{0\}}(X_1;\cF_{|_{X_1}}) &  & H^{0,\{1\}}(X_1;\cF_{|_{X_1}})
    \end{tikzcd}
    \end{equation}
    From the explicit computation of these cohomology spaces and the induced linear maps, we obtain the following interval decomposition:
    \[
    H^{0,\bullet}(X_1;\cF_{|_{X_1}})\simeq \R[\cP_1^{*,\op}]\oplus \R[\{0\}] \oplus \R[\{1\}]\, .
    \]
The direct image of each of these interval modules by the inclusion $i_{1}$ is given by the intervals (Equation~\eqref{Eq:IntervaloAsociado}):
\begin{align*}
    \widetilde{\cP^{*,\op}_1}&=\{\h\in \cP^{*,\op}_2 : \h\cap \{0,1\}\in \cP^{*,\op}_1\}=\cP^{*,\op}_2\smallsetminus\{2\}\,,\\[1.5ex]
    \widetilde{\{0\}}&=\{\h\in \cP^{*,\op}_2 :\h\cap \{0,1\}= \{0\}\}=\{\{0\},\{0,2\}\}\,,\\[1.5ex]
    \widetilde{\{1\}}&=\{\h\in \cP^{*,\op}_2 :\h\cap \{0,1\}= \{1\}\}=\{\{1\},\{1,2\}\}\,.
\end{align*}
Thus,
\begin{equation}\label{Eq:Descomposicioni1X1}
    {i_{1}}_*H^{0,\bullet}(X_1;\cF_{|_{X_1}})=\R[\cP^{*,\op}_2\smallsetminus\{2\}]\oplus \R[\{\{0\},\{0,2\}\}]\oplus \R[\{\{1\},\{1,2\}\}]\,.
\end{equation}

Let us now turn to the $2$-dimensional component, $X_2$. Since the structure of this module is significantly more complex, we detail the computation of its cohomology spaces and their respective bases. These details are necessary to identify $H^{0,\bullet}(X_2;\cF_{|_{X_2}})$ with the abstract persistence module analyzed in Example~\ref{Ex:Decomposition2}. To this end, we use the following notation:
\begin{itemize}
    \item $v_2$ for the lower-left vertex of the triangle,
    \item $v_3$ for the lower-right vertex,
    \item $v_4$ for the upper vertex,
    \item $e_{ij}$ for the edge joining $v_i$ and $v_j$,
    \item $\sigma_{234}$ for the $2$-simplex $\{v_2,v_3,v_4\}$,
    \item $\{x_i,y_i\}$ for the vectors of the canonical basis of $\cF(v_i)=\R^2$,
    \item $\{x_{23}\}$, $\{x_{24},y_{24}\}$, $\{x_{34},y_{34}\}$, $\{x_{234},y_{234}\}$, respectively, for the canonical bases of $\cF(e_{23})$, $\cF(e_{24})$, $\cF(e_{34})$, and $\cF(\sigma_{234})$.
\end{itemize}
 
To simplify the computations, we use the isomorphism given by Theorem~\ref{Thm:EquivalenciaCohomologiacelularyhaces}:
\[
H^{0,\{0,1,2\}}(X_2;\cF_{|_{X_2}})=H^{0}(P_{X_2};\widehat{\cF}_{|_{X_2}})\simeq H^0(X_2;\cF_{|_{X_2}})\,.
\]
By definition of cellular sheaf cohomology, $H^0(X_2;\cF_{|_{X_2}})$ is the kernel of the coboundary operator:
\[
\begin{array}{ccc}
  \R^2_{v_2}\oplus \R_{v_3}^2\oplus \R_{v_4}^2  & \stackrel{\delta^0}\longrightarrow& \R_{e_{23}}\oplus\R_{e_{24}}^2\oplus \R^2_{e_{34}}\\[1.5ex]
   ((\alpha_2,\beta_2),(\alpha_3,\beta_3),(\alpha_4,\beta_4))& \longmapsto & (\alpha_3+\beta_3-\alpha_2-\beta_2,\\  & & (\alpha_2-\alpha_4,\,\beta_{2}-\beta_4),\\& & (\alpha_3-\alpha_4,\,\beta_{3}-\beta_4)) 
\end{array}
\]
Thus,
\[
\begin{aligned}
    H^{0,\{0,1,2\}}(X_2;\cF_{|_{X_2}})&\simeq\{((\alpha_2,\beta_2),(\alpha_2,\beta_2),(\alpha_2,\beta_2))\in \R^2_{v_2}\oplus \R_{v_3}^2\oplus \R_{v_4}^2 : \alpha_2,\beta_2\in\R\}=\R^2\,.
\end{aligned}
\]
Let us now turn to the cohesion spaces associated with subsets of indices of the form $\{h_0,h_1\}$. First, note that $H^{0,\{0,1\}}(X_2;\cF_{|_{X_2}})\simeq H^{0,\{0,1,2\}}(X_2;\cF_{|_{X_2}})$. On the other hand, $H^{0,\{0,2\}}(X_2;\cF_{|_{X_2}})$ is the kernel of the linear map:
\[
\begin{array}{ccc}
    \R^2_{v_2}\oplus \R_{v_3}^2\oplus \R_{v_4}^2\oplus \R_{\sigma_{234}}^2& \stackrel{\delta^0}\longrightarrow &  \R^2_{v_2\lhd \sigma_{234}}\oplus \R^2_{v_3\lhd \sigma_{234}}\oplus\R^2_{v_4\lhd \sigma_{234}} \\[1.5ex]
      ((\alpha_2,\beta_2),(\alpha_3,\beta_3),(\alpha_4,\beta_4),(\alpha_{234},\beta_{234}))& \longmapsto & ((\alpha_{234}-\alpha_2-\beta_2,\,\beta_{234}),\\  & & (\alpha_{234}-\alpha_3-\beta_3,\,\beta_{234}),\\& & (\alpha_{234}-\alpha_4-\beta_4,\,\beta_{234}))
\end{array}
\]
Therefore,
\[
\begin{aligned}
    H^{0,\{0,2\}}(X_2;\cF_{|_{X_2}})&=\{((\alpha_2,\,\beta_2),\,(\alpha_2+\beta_2-\beta_3,\,\beta_3),\,(\alpha_2+\beta_2-\beta_4,\,\beta_4),\,(\alpha_2+\beta_2,\,0)) \}=\R^4\,.
\end{aligned}
\]
Next, the space $H^{0,\{1,2\}}(X_2;\cF_{|_{X_2}})$ is the kernel of the map:
\[
\begin{array}{ccc}
     \R_{e_{23}}\oplus \R_{e_{24}}^2\oplus\R_{e_{34}}^2\oplus\R_{\sigma_{234}}& \stackrel{\delta^0}\longrightarrow & \R^2_{e_{23}\face \sigma_{234}}\oplus \R^2_{e_{24}\face \sigma_{234}}\oplus\R^2_{e_{34}\face \sigma_{234}} \\[1.5ex]
    (\alpha_{23},\,(\alpha_{24},\,\beta_{24}),\,(\alpha_{34},\,\beta_{34}),\,(\alpha_{234},\,\beta_{234}))& \longmapsto & ((\alpha_{234}-\alpha_{23},\,\beta_{234}),\\  & &(\alpha_{234}-\alpha_{24}-\beta_{24},\,\beta_{234}),\\ & & (\alpha_{234}-\alpha_{34}-\beta_{34},\,\beta_{234}))
\end{array}
\]
and therefore,
\[
\begin{aligned}
    H^{0,\{1,2\}}(X_2;\cF_{|_{X_2}})&=\{(\alpha_{23},\,(\alpha_{23}-\beta_{24},\,\beta_{24}),\,(\alpha_{23}-\beta_{34},\,\beta_{34}),\,(\alpha_{23},\,0))\}=\R^3\,.
\end{aligned}
\]
Finally, the cohomology spaces associated with a single dimensional index are:
\[
\begin{aligned}
    H^{0,\{0\}}(X_2;\cF_{|_{X_2}})& =\langle x_2, y_2, x_3,y_3, x_4,y_4\rangle =\R^6\,,\\[1.5ex]
    H^{0,\{1\}}(X_2;\cF_{|_{X_2}})&=\langle x_{23}, x_{24}, y_{24}, x_{34}, y_{34}\rangle=\R^5\,,\\[1.5ex]
    H^{0,\{2\}}(X_2;\cF_{|_{X_2}})&=\langle x_{234}, y_{234}\rangle =\R^2\,.
\end{aligned}
\]

We have also obtained the following bases:
\begin{itemize}
    \item $\mathcal{B}_{012}=\{x_2+x_3+x_4,\, y_2+y_3+y_4\}$ of $H^{0,\{0,1,2\}}(X_2;\cF_{|_{X_2}})$,
    \item $\mathcal{B}_{01}=\{x_2+x_3+x_4,\, y_2+y_3+y_4\}$ of $H^{0,\{0,1\}}(X_2;\cF_{|_{X_2}})$,
    \item $\mathcal{B}_{02}=\{x_2+x_3+x_4+x_{234},\, y_2+x_3+x_4+x_{234},\,-x_3+y_3,\,-x_4+y_4\}$ of $H^{0,\{0,2\}}(X_2;\cF_{|_{X_2}})$,
    \item $ \mathcal{B}_{12}=\{ x_{23}+x_{24}+x_{34}+x_{234},\,-x_{24}+y_{24},\,-x_{34}+y_{34}\}$ of $H^{0,\{1,2\}}(X_2;\cF_{|_{X_2}})$,
    \item $\mathcal{B}_{0}=\{ x_2, y_2, x_3,y_3, x_4,y_4\}$ of $H^{0,\{0\}}(X_2;\cF_{|_{X_2}})$,
    \item $\mathcal{B}_{1}=\{x_{23}, x_{24}, y_{24}, x_{34}, y_{34}\}$ of $H^{0,\{1\}}(X_2;\cF_{|_{X_2}})$,
    \item $\mathcal{B}_{2}=\{x_{234}, y_{234}\}$ of $H^{0,\{2\}}(X_2;\cF_{|_{X_2}})$.
\end{itemize}

Relative to these bases, the cohesion persistence module $H^{0,\bullet}(X_2;\cF_{|_{X_2}})$ can be shown to have the same structure as the persistence module from Example~\ref{Ex:Decomposition2}. By the calculations in that example (Equation~\eqref{Eq:Decomposition2-Decomposition}), we have:
\begin{equation}\label{Eq:DecompositionH0X2}
\begin{split}
    H^{0,\bullet}(X_2;\cF_{|_{X_2}})\simeq &\, \mathbb{R}[\cP^{*,\op}_2]\,\oplus \,\mathbb{R}[\cP^{*,\op}_2\smallsetminus\{2\}]\, \oplus\, \mathbb{R}[\{\{0\},\{0,2\}\}]^2\,\oplus\\[1ex]
        & \mathbb{R}[\{\{1\},\{1,2\}\}]\,\oplus \,\mathbb{R}[\{0\}]^2\oplus \mathbb{R}[\{1\}]^2 \,\oplus\, \mathbb{R}[\{2\}].
\end{split}
\end{equation}

Substituting the computations of Equations~\eqref{Eq:DescomposicionX0},~\eqref{Eq:Descomposicioni1X1}, and~\eqref{Eq:DecompositionH0X2} into the isomorphism of Equation~\eqref{Eq:DescomposicionBase}, we obtain the decomposition of $H^{0,\bullet}(X;\cF)$ into interval modules represented in Figure~\ref{Fig:SheafPersistentCohesion-Barcode}:
\[
\begin{aligned}
    H^{0,\bullet}(X;\cF)\simeq & \ \R[\cP^{*,\op}_2]\,\oplus\,\R[\cP^{*,\op}_2\smallsetminus\{2\}]^2\oplus \R[\{\{0\},\{0,1\},\{0,2\},\{0,1,2\}\}]\\[1.5ex]
  & \oplus \R[\{\{0\},\{0,2\}\}]^3\oplus\R[\{\{1\},\{1,2\}\}]^2\oplus \R[\{0\}]^2\oplus\R[\{1\}]^2\oplus\R[\{2\}]\,.
\end{aligned}
\]

\end{example}
    \begin{figure}[htb!]
        \centering
        \includegraphics[width=0.95\linewidth]{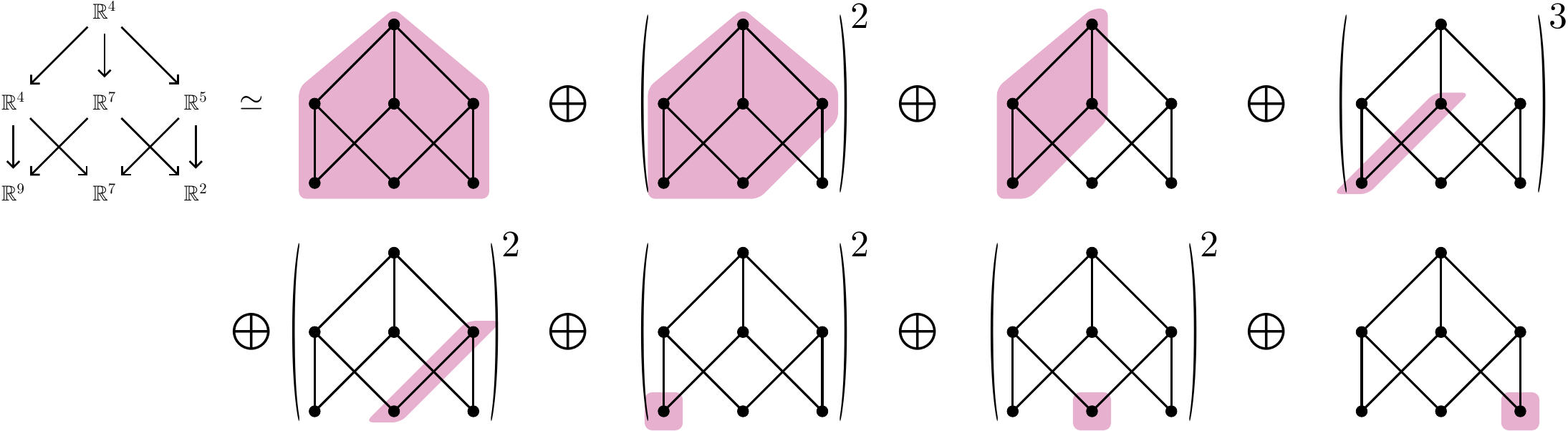}
        \caption{Interval decomposition of $H^{0,\bullet}(X;\cF)$ for the cellular sheaf in Figure~\ref{Fig:SheafPersistentCohesion}. }
        \label{Fig:SheafPersistentCohesion-Barcode}
    \end{figure}

\section{Sheaf theory and resilience}

We now extend the topological resilience framework of \cite{Pablo_Dani_Dario_25} to cellular sheaves. Given a degenerative process on a simplicial complex, we study how the algebraic information encoded by a sheaf persists when its simplicial support is degraded. The resulting constructions are biparameter persistence modules: one parameter describes the attack on the support, while the other records either a thickness or a cohesion threshold.

\subsection{Thickness resilience for sheaves} Given a cofiltration of simplicial complexes
\[    X^{\bullet}\colon
    X^0\supseteq X^1\supseteq\cdots\supseteq X^M,
\]
let $N\coloneqq\dim X^0$, and let $\cF$ be a cellular sheaf on $X^0$. By restricting $\cF$ to each stage of the cofiltration, we obtain the sheaf persistence module of topological type
\[
    H^n(X^\bullet;\cF_{|_{X^\bullet}})
    \colon [M]\longrightarrow\vect.
\]
This module records the persistence of sheaf cohomology under the structural degradation of the support. To incorporate thickness, we take coskeleta at every stage. Specifically, denote by $X^{i,q}\coloneqq (X^i)^q$ the $q$-coskeleton of $X^i$. The family $\{X^{i,q}\}_{(i,q)\in [M]\times [N]}$ forms a biparameter cofiltration starting at $X^{0,0}=X^0$. Restricting $\cF$ to these subcomplexes and taking cohomology gives a biparameter persistence module of topological type
\[
\begin{array}{lccl}
    H^{n,\bullet}_{\mathrm{th}}(X^\bullet;\cF) \colon &[M]\times [N]& \longrightarrow&\vect \\
    &(i,q)& \longmapsto&
    H^n(X^{i,q};\cF_{|_{X^{i,q}}}).
\end{array}
\]

For a fixed thickness threshold $q$, this biparameter module contains the ladder module
\[
\begin{tikzcd}
H^n(X^{0,q};\mathcal{F}_{|_{X^{0,q}}}) \arrow[r] 
& H^n(X^{1,q};\mathcal{F}_{|_{X^{1,q}}}) \arrow[r] 
& \dots \arrow[r] 
& H^n(X^{M,q};\mathcal{F}_{|_{X^{M,q}}}) \\
H^n(X^0;\mathcal{F}_{|_{X^0}}) \arrow[u, "\varphi^{n,0,q}"] \arrow[r] 
& H^n(X^1;\mathcal{F}_{|_{X^1}}) \arrow[u,"\varphi^{n,1,q}"] \arrow[r]
& \dots \arrow[r] 
& H^n(X^M;\mathcal{F}_{|_{X^M}}) \arrow[u, "\varphi^{n,M,q}"]
\end{tikzcd}
\]
The lower row measures the persistence of the sheaf cohomology along the degradation process, while the upper row measures the persistence of its $q$-thickness cohomology. The image persistence module
\begin{equation*}
        \img \varphi^{n,\bullet,q} \colon \img \varphi^{n,0,q} \to \img \varphi^{n,1,q} \to \cdots \to \img \varphi^{n,M,q}
    \end{equation*}
records the classes that remain supported on simplices of dimension at least $q$; the kernel 
 \begin{equation*}
        \ker \varphi^{n,\bullet,q} \colon \ker \varphi^{n,0,q} \to \ker \varphi^{n,1,q} \to \cdots \to \ker \varphi^{n,M,q}
    \end{equation*}
records classes lost after passing to the $q$-coskeleta; and the cokernel 
\begin{equation*}
        \coker \varphi^{n,\bullet,q} \colon \coker \varphi^{n,0,q} \to \coker \varphi^{n,1,q} \to \cdots \to \coker \varphi^{n,M,q}
  \end{equation*}
records classes created by this restriction.

\subsection{Cohesion resilience for sheaves}

We now incorporate cohesion. For each stage $X^i$ and each set of dimensions $\h\in\cP^{*}_N$, consider the $\h$-face poset $P_{X^i}^{\h}$. These posets form a cofiltration of finite topological spaces indexed by $[M]\times\cP^{*,\op}_N$, whose initial term is $P_{X^0}^{[N]}=P_{X^0}$. Restricting the sheaf $\widehat{\cF}$ to these finite topological spaces and applying sheaf cohomology yields a biparameter persistence module
\[
\begin{array}{rccl}
    H^{n,\bullet}_{\mathrm{coh}}(X^\bullet;\cF) \colon &[M]\times\cP^{*,\op}_N& \longrightarrow&\vect \\
    &(i,\h)&\longmapsto &H^n(P_{X^i}^{\h};\widehat{\cF}_{|_{\scriptstyle{P_{X^i}^{\h}}}})
\end{array}
\]
Fixing a cohesion index $\h$, the inclusions $P_{X^i}^{\h}\subseteq P_{X^i}$ induce a ladder module
\[
\begin{tikzcd}
H^{n}(P_{X^0}^\h;\widehat{\cF}_{|_{\scriptstyle P_{X^0}^\h}})  \arrow[r]
& H^{n}(P_{X^1}^{\h};\widehat{\cF}_{|_{\scriptstyle P_{X^1}^\h}})  \arrow[r]
& \dots \arrow[r]
& H^n(P_{X^M}^\h;\widehat{\cF}_{|_{\scriptstyle P_{X^M}^\h}})  \\
H^n(P_{X^0};\widehat{\cF}_{|_{\scriptstyle P_{X^0}}}) \arrow[u, "\varphi^{n,0,\h}"] \arrow[r]
& H^n(P_{X^1};\widehat{\cF}_{|_{\scriptstyle P_{X^1}}}) \arrow[u, "\varphi^{n,1,\h}"] \arrow[r]
& \dots \arrow[r]
& H^n(P_{X^M};\widehat{\cF}_{|_{\scriptstyle P_{X^M}}}) \arrow[u, "\varphi^{n,M,\h}"]
\end{tikzcd}
\]
By Theorem~\ref{Thm:EquivalenciaCohomologiacelularyhaces}, the lower row is the sheaf cohomology persistence module along the original degradation process. The upper row measures the persistence of $\h$-cohesion cohomology. The image module $\img\varphi^{n,\bullet,\h}$ tracks the classes that are completely determined by the $\h$-dimensional simplices; the kernel $\ker \varphi^{n,\bullet,\h}$ records classes destroyed by ignoring the remaining dimensions; and the cokernel $\coker \varphi^{n,\bullet,\h}$ captures classes created by this selective deletion. 

\subsection{An illustrative example of sheaf data under topological degradation}

    Let $X^0$ denote the triangulated annulus in Figure~\ref{Fig:Filtraciondeataques}A, and let $\cF$ be the rank-$1$ cellular sheaf over $X^0$ illustrated in Figure~\ref{Fig:ResilienciaHacesHaz}. Consider the cofiltration $X^0\supset X^1 \supset X^2$, represented in Figure~\ref{Fig:Filtraciondeataques}, which we regard as an attack cofiltration, or as a degenerative process of $X^0$, where:
    \begin{itemize}
\item \(X^0\) is the full triangulated annulus,
\item \(X^1=X^{0}\smallsetminus\{\sigma_{015},\,\sigma_{145},\,e_{14},\,e_{15}\}\) is an annulus with a $1$-dimensional seam defined by the edge~$e_{05}$,
\item \(X^2=X^{1}\smallsetminus\{e_{05}\}\).
\end{itemize}
\begin{figure}[htb!]
    \centering
    \begin{subfigure}{0.32\textwidth}
        \centering
        \includegraphics[width=0.95\linewidth]{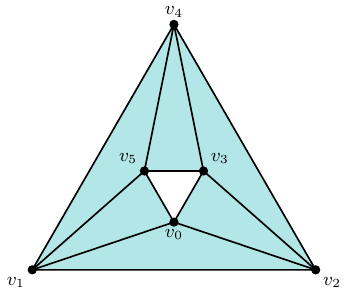}
        \caption{$X^0$}\label{Fig:FiltraciondeataquesA}
    \end{subfigure}
    \hfill
    \begin{subfigure}{0.32\textwidth}
        \centering
        \includegraphics[width=0.95\linewidth]{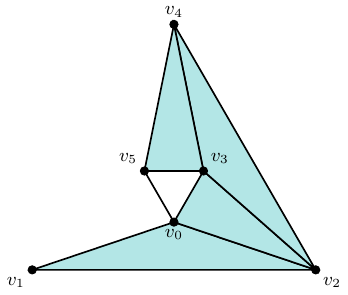}
        \caption{$X^1$}
    \end{subfigure}
    \hfill
    \begin{subfigure}{0.32\textwidth}
        \centering
        \includegraphics[width=0.95\linewidth]{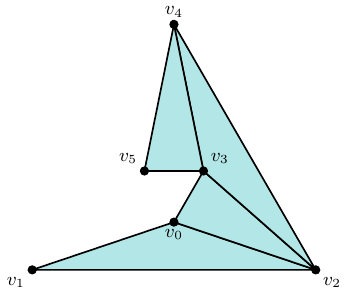}
        \caption{$X^2$}
    \end{subfigure}
    \caption{Attack cofiltration $X^\bullet$.}
    \label{Fig:Filtraciondeataques}
\end{figure}

\begin{figure}
    \centering
    \includegraphics[width=0.65\linewidth]{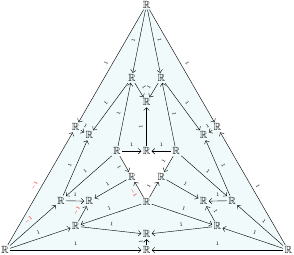}
    \caption{Cellular sheaf $\cF$ on the simplicial complex $X^0$ of Figure~\ref{Fig:Filtraciondeataques}A.}
    \label{Fig:ResilienciaHacesHaz}
\end{figure}

The cofiltration of $1$-coskeleta is exactly the original cofiltration, since it has no isolated vertices. For this reason, we focus on the study of thickness in dimension~$2$. The $2$-thickness cofiltration associated with $X^\bullet$ is:
\[
X^{0,2}=X^0\supseteq X^{1,2}=X^2\supseteq X^{2,2}=X^2\,.
\]
That is, we have the following diagram:
\[
\begin{tikzcd}[column sep=3.5em, row sep=3.5em]
X^{0,2}=X^0 & X^{1,2}=X^2\arrow[d, hook'] \arrow[l, hook'] & X^{2,2}=X^2 \arrow[l, equals] \\
X^0 \arrow[u, equals] & X^1 \arrow[l, hook']  & X^2 \arrow[l, hook'] \arrow[u, equals]
\end{tikzcd}
\]
Restricting the sheaf $\cF$ to each subcomplex in the preceding diagram and passing to cohomology yields the following diagrams:
\[
\begin{tikzcd}[column sep=3.5em, row sep=3.2em]
H^0(X^0;\mathcal F) \arrow[r] &
H^0(X^2;\mathcal F_{|_{X^2}}) \arrow[r] &
H^0(X^2;\mathcal F_{|_{X^2}}) \\
H^0(X^0;\mathcal F) \arrow[u, "\varphi^{0,0,2}"] \arrow[r] &
H^0(X^1;\mathcal F_{|_{X^1}}) \arrow[u, "\varphi^{0,1,2}"] \arrow[r] &
H^0(X^2;\mathcal F_{|_{X^2}}) \arrow[u, "\varphi^{0,2,2}"]
\end{tikzcd}
\]
\[
\begin{tikzcd}[column sep=3.5em, row sep=3.2em]
H^1(X^0;\mathcal F) \arrow[r] &
H^1(X^2;\mathcal F_{|_{X^2}}) \arrow[r] &
H^1(X^2;\mathcal F_{|_{X^2}}) \\
H^1(X^0;\mathcal F) \arrow[u, "\varphi^{1,0,2}"] \arrow[r] &
H^1(X^1;\mathcal F_{|_{X^1}}) \arrow[u, "\varphi^{1,1,2}"] \arrow[r] &
H^1(X^2;\mathcal F_{|_{X^2}}) \arrow[u, "\varphi^{1,2,2}"]
\end{tikzcd}
\]
Proceeding as in Section~\ref{ss:CohHaz}, we compute the corresponding cohomology spaces and their induced linear maps, obtaining the diagrams below (degree $0$ on the left, degree $1$ on the right): 
\begin{equation}\label{e:resgro}
\begin{tikzcd}[column sep=3.8em, row sep=3.2em]
0 \arrow[r] &
\mathbb R \arrow[r] &
\mathbb R
&&
0 \arrow[r] &
0 \arrow[r] &
0
\\
0 \arrow[u, "0"] \arrow[r] &
0 \arrow[u, "0"] \arrow[r] &
\mathbb R \arrow[u, "\mathrm{Id}"]
&&
0 \arrow[u, "0"] \arrow[r] &
0 \arrow[u, "0"] \arrow[r] &
0 \arrow[u, "0"]
\end{tikzcd}
\end{equation}
Since the diagram associated with cohomology in degree $1$ is zero, we focus on what remains in the diagram associated with degree $0$. The kernel, image, and cokernel persistence modules of the diagram on the left are:
\[
\ker \varphi^{0,\bullet,2} \colon 0\longrightarrow 0\longrightarrow 0 \qquad \quad \img \varphi^{0,\bullet,2} \colon 0\longrightarrow 0\longrightarrow \mathbb R \qquad \quad \coker \varphi^{0,\bullet,2} \colon 0\longrightarrow \R \longrightarrow 0
\]

The sheaf structure exhibits local compatibility within each triangle, as the restriction of global sections to any individual triangle in $X^0$ is nonzero. Nevertheless, a global incompatibility arises when going around the central hole. To interpret this, consider the following cycle around the hole:
\[
\gamma\colon v_0\to v_1\to v_2\to v_3\to v_4\to v_5 \to v_0\,.
\]
For each $0\leq i,j\leq 5$, write $T_{i,j}=\cF_{v_j\face e_{ij}}^{-1}\cF_{v_i\face e_{ij}}$. Then, along the cycle $\gamma$,
\[
T_{\gamma}=T_{5,0}\,T_{4,5}\,T_{3,4}\,T_{2,3}\,T_{1,2}\,T_{0,1}=1\cdot 1\cdot 1\cdot 1\cdot 1\cdot (-1)=-1\,.
\]
This negative sign is invariant along any cycle surrounding the central hole. This is precisely the global obstruction that prevents the existence of global sections of the sheaf on $X^0$ and $X^1$. From all of the above, we make the following interpretation for the cohomology of the sheaf $\cF$:
\begin{enumerate}
\item The type-T persistence of the sheaf $\mathcal F$ for the cofiltration $X^\bullet$ tracks the evolution of the data defined by the sheaf over the attack cofiltration. This is given by the lower rows of the diagrams~\eqref{e:resgro}, and the barcode in degree $0$ has a single interval, $[2,3)$. In this example, this tells us that:

\begin{itemize}[leftmargin=2em]
\item In \(X^0\) and \(X^1\), there are no nontrivial cohomology classes in degree \(0\), since the cycle $\gamma$ around the annulus still transports the sign \(-1\), which prevents a global section from closing coherently.
\item Nontrivial global sections emerge at attack step $2$, in \(X^2\), when the cycle breaks and the global obstruction disappears.
\item Consequently, in this case the persistent cohomology of the sheaf detects when the global obstruction in the filtration $X^\bullet$ vanishes, but does not distinguish whether that obstruction is supported in a thick or thin part of the system.
\end{itemize}

\item The persistence of the \(2\)-thickness (type-T persistence of the sheaf $\mathcal F$ applied to the cofiltration of $2$-coskeleta) is given by the upper rows of the diagrams~\eqref{e:resgro}. The barcode of the persistence module associated with cohomology in degree $0$ consists of the interval $[1,3)$. This persistence captures finer geometric information than the previous one:
\begin{itemize}
\item Already at the first attack there is a thick \(2\)-dimensional part of the complex, $X^{1,2}=X^{2}$, which does not contain the global obstruction, since we have removed the edge $e_{05}$ and all remaining restriction morphisms are the identity.
\item It is precisely because of the preceding fact that a global section appears one stage earlier. This early appearance confirms that the initial obstruction was neither thick nor robust, as it is not supported on $X^{1,2}$.
\end{itemize}
\item The barcode of the cokernel persistence module, $\coker \varphi^{0,\bullet,2}$, consists of the interval $[1,2)$. At each stage, this keeps track of the global sections over the $2$-coskeleton that do not extend to global sections of the total complex. It has an added predictive capacity. In this example:
\begin{itemize}
    \item The thick part, \(X^{1,2}=X^2\), has a global section since \(H^0(X^{2};\mathcal F_{|_{X^2}})=\mathbb R\), yet this class cannot be extended to $X^1$. Thus, the obstruction to the existence of a global section is no longer in the thick part of the system, but in the difference $X^1\smallsetminus X^{1,2}$, contributing the sign change in the transport $T_{\gamma}$ along any cycle surrounding the central hole.
    \item The cokernel module not only detects a discrepancy between the cohomology of the thick part $X^{\bullet,2}$ and that of the total complex $X^\bullet$, but also locates the stage at which the obstruction moves from being supported in the thick part to being supported in a thin region (the edge $e_{05}$ of $X^1$).
\end{itemize}
\end{enumerate}

\section*{Acknowledgements} 
The authors would like to thank Fernando Sancho de Salas for his insightful and helpful comments. This work is supported by Spanish National Grant PID2021-128665NB-I00 funded by MCIN/AEI/10.13039/501100011033 and, as appropriate, by ``ERDF A way of making Europe''; and also by project STAMGAD 18.J445/463AC03 by Consejer\'ia de Educaci\'on (GIR, Junta de Castilla y Le\'on, Spain).  Pablo Hern\'andez-Garc\'ia also acknowledges financial support from the USAL 2022 call for predoctoral contracts, co-financed by Banco Santander.

\bibliographystyle{abbrv}
\bibliography{biblio}

\end{document}